\documentclass[11p]{article}
\usepackage[a4paper, margin=1in]{geometry}

\usepackage{amsmath}
\usepackage{amssymb}
\usepackage{mathtools}
\usepackage{makecell}
\usepackage{bbm}
\usepackage{bm}
\usepackage{subcaption}
\usepackage{tikz}
\usepackage{hyperref}
\usepackage{amsthm}
\newtheorem{theorem}{Theorem}

\newtheorem{proposition}{Proposition}
\newtheorem{corollary}{Corollary}

\newtheorem{assumption}{Assumption}
\theoremstyle{definition}
\newtheorem{remark}{Remark}

\allowdisplaybreaks
\usepackage{algpseudocode}
\usepackage{algorithm}
\algrenewcommand\algorithmicrequire{\textbf{Input:}}
\algrenewcommand\algorithmicensure{\textbf{Output:}}
\usepackage{url}

\usepackage{comment}
\usepackage[colorinlistoftodos]{todonotes}
\usepackage{cleveref}

\usepackage{booktabs}
\usepackage{multirow}
\usepackage{siunitx} 
\usepackage{upgreek}
\usepackage{soul} 
\usepackage{setspace}
  
\makeatletter
\newcommand{\leqnomode}{\tagsleft@true}
\newcommand{\reqnomode}{\tagsleft@false}
\makeatother

\newcommand{\tr}{\textup{tr}}
\newcommand{\Diag}{\textup{Diag}}
\newcommand{\diag}{\textup{diag}}

\newcommand{\Conv}{\textup{conv}}

\newcommand{\rank}{\textup{rank}}

\newcommand{\stirlingii}{\genfrac{\{}{\}}{0pt}{}}

\begin{document}

\title{Lagrangian Duality for Mixed-Integer Semidefinite Programming: Theory and Algorithms}
\author{Frank de Meijer \thanks{Delft Institute of Applied Mathematics, Delft University of Technology, The Netherlands, {\tt f.j.j.demeijer@tudelft.nl}}
	\and {Renata Sotirov}  \thanks{CentER, Department of Econometrics and OR, Tilburg University, The Netherlands, {\tt r.sotirov@uvt.nl}}}

\date{\today}

\maketitle

\begin{abstract}

This paper presents the Lagrangian duality theory for mixed-integer semidefinite programming (MISDP). We derive the Lagrangian dual problem and prove that the resulting Lagrangian dual bound dominates the  bound obtained from the continuous relaxation of the MISDP problem.  We  present a hierarchy of Lagrangian dual bounds by exploiting the theory of integer positive semidefinite matrices and propose three algorithms for obtaining those bounds. Our algorithms are variants of well-known algorithms for minimizing non-differentiable convex functions. 
The numerical results on the max-$k$-cut problem show that the Lagrangian dual bounds 
are substantially stronger than the semidefinite programming bound obtained by relaxing integrality, already for lower levels in the hierarchy.

\end{abstract}


\section{Introduction}

Mixed-integer semidefinite programming (MISDP)  may be seen as a generalization of mixed-integer linear programming (MILP), where the vector of variables is replaced by a positive semidefinite matrix variable in which some of the entries are integer and others real.
In integer semidefinite programming  (ISDP) all entries of the positive semidefinite matrix variable are required to be integer.

The combination of positive semidefiniteness and integrality allows to formulate various optimization problems as  mixed-integer semidefinite programs (MISDPs).
Since the 1990s, a few papers appeared
that present ISDP formulations for classical  discrete optimization problems e.g.,~\cite{Cvetkovic,Eisenblatter,Meurdesoif2005},
and several MISDP formulations of applied optimization problems, see e.g.,~\cite{GilGonzalezEtAl,YonekuraKanno,ZhengEtAl}.
In~\cite{DeMeijerSotirov23}, the authors introduce a generic approach for deriving MISDP formulations of binary quadratically constrained quadratic programs and binary quadratic matrix programs. They also show that  several optimization problems allow for novel MISDP formulations, which 
provide  new perspectives on solution approaches for those problems.

The MISDPs  can be used to  compute exact solutions for the optimization problems using generic MISDP solvers, e.g., branch-and-bound algorithms \cite{GallyEtAl, MatterPfetsch} or branch-and-cut algorithms \cite{deMeijerSotirovCG,KobayashiTakano}. 
Another solver that supports solving MISDPs while exploiting sparsity is GravitySDP~\cite{Hijazi}.
Numerical tests in those and related papers   show that solving an MISDP formulation of an optimization problem might be more beneficial than  solving 
a mixed-integer programming  (MIP)  formulation of the same problem via state-of-the-art MIP solvers.

A crucial feature of an efficient solver is being able to compute strong bounds fast.
Most of the current solving approaches for MISDPs  solve SDP  relaxations in each node of a branching tree. 
Since a generic way of solving an SDP relaxation is an interior point method, computational effort in each node of a branching tree may be (very) high.
It is well-known that interior point-based SDP solvers exhibit problems in terms of both time and memory for solving even medium-size semidefinite programs (SDPs). 
Moreover, interior point methods have difficulties with handling additional cutting planes,  including nonnegativity constraints that significantly strengthen SDP bounds. There exist  alternative approaches for solving SDP relaxations based on Alternating Direction Augmented Lagrangian (ADAL) methods, see e.g.,~\cite{HuSotirov,SunTohetAl,OliveiraEtAl}.
Those algorithms are first-order methods that have low memory requirement, but might suffer a tailing-off effect when there is need for a high precision solution. However, ADAL methods can handle a large number of  cutting planes together with semidefiniteness efficiently~\cite{deMeijeretAll24,DeMeijerSotirov2,DeMeijerEtAl}. 
The bounds resulting from relaxing a MISDP to an SDP
do not exploit integrality of the variables. The goal of this work is to provide alternative bounding approaches that do exploit both positive semidefiniteness and integrality.

In this paper we derive the Lagrangian duality theory for MISDP that enables us to derive a hiearchy of  Lagrangian dual bounds, and show how to compute those bounds using first-order methods. 
There exist results on superadditive duality theory for mixed-integer conic optimization~\cite{doi:10.1137/18M1210812}.  
However, the subadditive dual for conic mixed-integer programs does not yield straightforward solution procedures, while the Lagrangian duality  theory for MISDP
does.

\subsection*{Main results and outline}

We first introduce  the Lagrangian duality theory for MISDP. 
We derive a Lagrangian relaxation of the MISDP problem by dualizing the positive semidefinite (PSD) constraint and a part of the linear constraints. 
The aim is to dualize  constraints that are intractable in combination with the integrality constraints.
From the Lagrangian relaxation we derive the Lagrangian dual problem.
We relate the optimal value of the continuous relaxation of the MISDP problem, the  Lagrangian dual bound, 
and the optimal value of the MISDP problem within a  sandwich theorem, showing that the Lagrangian dual bound is always at least as strong as the bound obtained from the continuous SDP relaxation. We also derive conditions under which equalities throughout the sequence of the sandwich theorem follow.

For the case of purely integer SDPs, we introduce a hierarchy of  Lagrangian dual bounds by exploiting the theory of integer PSD matrices. Namely, we propose  partitioning a matrix variable into submatrices of  prescribed sizes and solving the Lagrangian dual problem  by exploiting finite generating sets of integer PSD matrices. The larger the size of the matrices in the partitions is, the stronger the Lagrangian dual bounds become. 
The resulting bounds are related to the so-called exact subgraph approach considered in~\cite{AdamsAnjosRendlWiegele,RendlGaar}.
In fact,  exact subgraph bounds are  Lagrangian dual bounds resulting from appropriate~ISDP models.

We propose three algorithms for computing the Lagrangian dual bounds.
We consider a variant of the deflected subgradient algorithm with the update scheme from~\cite{SheraliUlular}  
and Polyak’s stepsize~\cite{Polyak}. We also consider a version of the accelerated gradient method introduced by Nesterov~\cite{Nesterov1983AMF},
and a  version of the proximal bundle algorithm, see e.g.,~\cite{Kiwiel}. All three variants include projections onto the PSD cone to preserve dual feasibility in each iteration of
the algorithm, which distinguishes them from their classical versions. The subgradient methods were not used, up to date, for computing SDP bounds.
Variants of the proximal  bundle method have been used for obtaining  SDP bounds  of several optimization problems, see e.g.,~\cite{Fischer2006ComputationalEW,RendlGaar,Rendl2007BoundsFT}. 
The Lagrangian relaxations in these existing algorithms require solving an SDP relaxation in each bundle iteration. Instead, we compute the Lagrangian relaxations  over a discrete set. 

We evaluate our algorithms on the max-$k$-cut problem for $k\in\{2,3,4\}$. 
The Lagrangian dual bounds turn out to be stronger than the SDP bounds obtained from the continuous relaxations of MISDPs, already for lower levels in the hierarchy, however the bounds keep on improving for higher levels in the hierarchy. This work provides an alternative approach for computing bounds of combinatorial optimization problems via first order methods by exploiting both positive semidefiniteness and integrality.

 \medskip

 This paper is structured as follows.  Lagrangian duality theory for MISDPs is introduced in Section~\ref{section:Lagrangian duality Theory}.
A hierarchy of Lagrangian dual bounds for ISDPs is derived in Section~\ref{sect:Hierarchy of Lagrangian dual}.
In Section~\ref{Sect:SolvingLagDual} we present three approaches for solving the Lagrangian dual problem. In particular, in Section~\ref{Subsec:ProjectedSubgradient}
we present a projected-deflected  subgradient method,  in Section~\ref{Subsec:Nesterov} a projected-accelerated subgradient method, and in
Section~\ref{Subsec:Projectedbundle} a projected bundle method. In Section~\ref{Sect:MaxKCut} we present an ISDP formulation of the max-$k$-cut problem, after which numerical results on this problem are provided in Section~\ref{Subsec:MaxCutNumerics}. Section~\ref{sect:conclusion} concludes our work.

\subsection*{Notation}
For $n \in \mathbb{Z}_+$, we define the set $[n] := \{1, \ldots, n\}$.
We denote by $\mathbf{0}_n \in \mathbb{R}^n$  the vector of all zeros, and by $\mathbf{1}_n \in \mathbb{R}^n$ the vector of all  ones. 
The  matrix of all-ones of size $n\times n$ is denoted by $\mathbf{J}_n$.
We omit the subscripts of these matrices when there is no confusion about the size.
The Euclidean norm of vector $x$ is denoted by $||x||$.
For any matrix $X \in \mathbb{R}^{n \times n}$ and indicator set $S \subseteq [n]$, we let $X[S]$ denote the principal submatrix induced by $S$.

We denote the set of all $n \times n$ real symmetric matrices by  $\mathcal{S}^n$.
The cone of symmetric positive semidefinite matrices is defined as $\mathcal{S}^n_+ := \{ {X} \in \mathcal{S}^n \, : \, \, {X} \succeq \mathbf{0} \}$. Thus, ${X} \succeq \mathbf{0}$ means  that $X$ is a positive semidefinite  matrix. 
The trace of a square matrix ${X}=(X_{ij})$ is given by $\tr({X})=\sum_{i}X_{ii}$.
 For any ${X},{Y} \in \mathcal{S}^n$ the trace inner product is defined as $\langle {X}, {Y} \rangle := \tr({X} {Y}) = \sum_{i,j = 1}^n X_{ij} Y_{ij}$. The associated norm is the Frobenius norm, denoted by $|| X ||_F := \sqrt{\tr(X^\top X)}$.

 The operator $\diag : \mathbb{R}^{n \times n} \rightarrow \mathbb{R}^n$ maps a square matrix to a vector consisting of its diagonal elements. We denote by $\Diag : \mathbb{R}^n \rightarrow \mathbb{R}^{n \times n}$ its adjoint operator.

\section{Lagrangian duality theory for MISDP}
\label{section:Lagrangian duality Theory}

In this section we extend the Lagrangian duality theory from mixed-integer linear programming to the case of mixed-integer semidefinite programming. 
 The results presented in this section are similar to those from the MILP literature, see e.g.,~\cite{Geoffrion}. In particular, we establish weak duality and a sandwich theorem for the Lagrangian dual bounds,  among other results. 

Let $C \in \mathcal{S}^n$, $b \in \mathbb{R}^m$ and let $\mathcal{A}: \mathcal{S}^n \rightarrow \mathbb{R}^m$ be a linear operator that is defined by $\mathcal{A}(X)_i := \langle A_i, X \rangle$ with $A_i \in \mathcal{S}^n$ for all $i \in [m]$. We define $\mathcal{A}^* : \mathbb{R}^m \rightarrow \mathcal{S}^n$ to be its adjoint.
Moreover, we let $\mathcal{J} \subseteq [n] \times [n]$ be an index set of the integer variables in the program. For each $(i,j) \in \mathcal{J}$, the set $B_{ij} \subseteq \mathbb{Z}$ denotes the integer solution space of the variable indexed by $(i,j)$. Finally, for all $(i,j) \in \mathcal{J}$, let $l_{ij}\in \mathbb{Z}$ and $u_{ij} \in \mathbb{Z}$ denote the lower and upper bound, respectively, with respect to the set $B_{ij}$. We assume these lower and upper bounds to be finite, implying that the sets $B_{ij}$ are finite. We consider a MISDP problem in the following general form: 
\begin{align} \label{ISDP} \tag{$MISDP$}
&\begin{aligned}
z_{MISDP} := \quad \min \quad & \langle C, {X} \rangle  \\
\text{s.t.} \quad &  \mathcal{A}(X) = b,~ X \succeq \mathbf{0} \\
& X_{ij} \in B_{ij} \qquad \text{for all } (i,j) \in \mathcal{J}.  \end{aligned} 
\intertext{{Throughout this work, we presume that the problem~\eqref{ISDP} is feasible, while we briefly revisit the issue of feasibility at the end of this section, see Remark~\ref{Remark:Feasibility}.} By relaxing the integrality constraints, we obtain the continuous SDP relaxation of \eqref{ISDP}:}
\label{SDP} \tag{$SDP$}
&\begin{aligned}
\,\,\,\,\,\,\, z_{SDP} := \quad \min \quad & \langle C, {X} \rangle  \\
\text{s.t.} \quad &  \mathcal{A}(X) = b,~ X \succeq \mathbf{0} \\
& l_{ij} \leq X_{ij} \leq u_{ij} \qquad \text{for all } (i,j) \in \mathcal{J}.\end{aligned} 
\end{align}
The inequality $z_{SDP} \leq z_{MISDP}$ clearly holds. Throughout this section, we make the following assumption, which is natural for many MISDPs originating from discrete optimization~\cite{DeMeijerSotirov23}.
\begin{assumption} \label{As:bounded}
The feasible set of~\eqref{ISDP} is bounded.
\end{assumption}
Since all integer variables are bounded by $l_{ij}$ and $u_{ij}$, respectively, Assumption~\ref{As:bounded} implies that bounds on the continuous variables are enforced by the constraints of~\eqref{ISDP}. 
Observe that the boundedness assumption on~\eqref{ISDP} also implies that a solution to the continuous SDP is attained. 
Namely, if the feasible set of~\eqref{SDP} would be unbounded, then there exists a ray $R \in \mathcal{S}^n$ with $\mathcal{A}(R) = \bold{0}$ and $R \succeq \bold{0}$. Since all bounds on variables in $\mathcal{J}$ are finite, $R_{ij} = 0$ for all $(i,j) \in \mathcal{J}$. Therefore, $R$ would also be a ray of the feasible set of~\eqref{ISDP}, contradicting Assumption~\ref{As:bounded}.

The Lagrangian dual of~\eqref{ISDP} is obtained by dualizing  the constraints that are intractable in combination with the integrality constraints. In our setting, it is natural to dualize the constraint $X \succeq \mathbf{0}$. Moreover, we can distinguish between the equalities in $\mathcal{A}(X) = b$ that are tractable with the integrality constraints and the equalities that are not. Consequently, we split $\mathcal{A}(X) = b$ into $\mathcal{A}_1(X) = b_1$ and $\mathcal{A}_2(X) = b_2$ where $b_1 \in \mathbb{R}^{m_1}$ and $b_2 \in \mathbb{R}^{m_2}$ with $m_1 + m_2 = m$. Here we assume that the equalities $\mathcal{A}_1(X) = b_1$ are not tractable in combination with the integrality constraints. We define the Lagrangian $\mathcal{L}(\cdot)$ after dualizing the constraints $X \succeq \mathbf{0}$ and $\mathcal{A}_1(X) = b_1$ as
\begin{align*}
    \mathcal{L}(X, S, \lambda) := \langle C, X \rangle - \langle S, X \rangle + \lambda^\top \left( \mathcal{A}_1(X) - b_1 \right),
\end{align*}
where $S \succeq \mathbf{0}$ and $\lambda \in \mathbb{R}^{m_1}$ are the corresponding Lagrange multipliers.  Moreover, we let $P$ denote 
the set of mixed-integer  symmetric matrices induced by the remaining constraints, i.e.,
\begin{align} \label{def:setP} 
    P := \left\{ X \in \mathcal{S}^n \, : \, \, \mathcal{A}_2(X) = b_2,~ X_{ij} \in B_{ij}~\text{for all } (i,j) \in \mathcal{J}  \right \}.
\end{align}
Without loss of generality, we may assume that $P$ is bounded. Namely, if not, we can add to $\mathcal{A}_2(X) = b_2$ the variable bounds on the continuous variables (which exist due to Assumption~\ref{As:bounded}).
Let us define the Lagrangian dual function $g: \mathcal{S}^n_+ \times \mathbb{R}^{m_1} \rightarrow \mathbb{R}$ as follows: 
\begin{align} \label{def:dual_function}\tag{$LR(S,\lambda)$}
    g(S,\lambda) := \min \left\{ \mathcal{L}(X, S, \lambda) \, : \, \, X \in P \right\}. 
\end{align}
Obviously, for all $S \in \mathcal{S}^n_+, \lambda \in \mathbb{R}^{m_1}$ we have $g(S,\lambda) \leq \mathcal{L}(X^*,S, \lambda) \leq \langle C, X^* \rangle = z_{MISDP}$, where $X^*$ is an optimal solution to~\eqref{ISDP}. We call the minimization problem  \eqref{def:dual_function} the {\em Lagrangian relaxation of} \eqref{ISDP} {\em with parameters} $(S,\lambda)$. To obtain the best lower bound for $z_{MISDP}$, we take the supremum of $g(S,\lambda)$ with respect to the dual variables $S$ and $\lambda$. This leads to the \emph{Lagrangian dual} of \eqref{ISDP}: 
\begin{align} \label{Eq:LD} \tag{$LD$}
\begin{aligned}
z_{LD} := \quad \sup \quad & g(S,\lambda)   \\
\text{s.t.} \quad &  S \succeq \mathbf{0},~ \lambda \in \mathbb{R}^{m_1}. \end{aligned} 
\end{align}
The following result follows by construction. 

\begin{proposition}[Weak duality] \label{Thm:weakduality}
$z_{LD} \leq z_{MISDP}$
\end{proposition}
An optimal solution of  the  Lagrangian relaxation 
\eqref{def:dual_function} may be an optimal solution to \eqref{ISDP}, as stated in the following proposition.
\begin{proposition}
Let $S \succeq \mathbf{0}$ and $\lambda \in \mathbb{R}^{m_1}$ be given. If $X^* \in {\mathcal S}^n$ is an optimal solution for \eqref{def:dual_function} that is feasible for \eqref{ISDP} and satisfies the complementarity slackness condition $SX^*=0$, then 
$X^*$ is optimal for \eqref{ISDP}.
\end{proposition}
In the sequel, we show that it is possible to obtain $z_{LD}$ as the solution of a continuous semidefinite programming problem. We consider
\begin{align} \label{EquivalentLD}
\begin{aligned}
\min \quad & \langle C , X \rangle   \\
\text{s.t.} \quad & \mathcal{A}_1(X) = b_1, ~ X \succeq \mathbf{0} \\
&  X \in \Conv\left( P \right). \\
\end{aligned} 
\end{align}
Observe that $P$ is defined as the set of mixed-integer points contained in a polyhedron.
Since the feasible set of~\eqref{EquivalentLD} is contained in the feasible set of~\eqref{SDP} and the latter one is compact, it follows that an optimal solution to~\eqref{EquivalentLD} is attained.
We show below that the optimization problem \eqref{EquivalentLD} is equivalent to the Lagrangian dual of \eqref{ISDP} based on a similar result for MILP by Geoffrion~\cite{Geoffrion}.
\begin{theorem} \label{Thm:AlternativeLD} 
Suppose Assumption~\ref{As:bounded} holds and let $\hat{z}$ denote the optimal objective value to~\eqref{EquivalentLD}, provided that it exists. Then $z_{LD} = \hat{z}$.
\end{theorem}
\begin{proof}
Since $\mathcal{L}(\cdot, S, \lambda)$ is linear on $\mathcal{S}^n$ for all fixed $S \in \mathcal{S}^n_+$ and $\lambda \in \mathbb{R}^{m_1}$, we have
\begin{align}
    z_{LD} & = \sup_{S \succeq \mathbf{0}, \lambda} \min_{X} \left\{ \langle C, X \rangle - \langle S, X \rangle + \lambda^\top \left( \mathcal{A}_1(X) - b_1 \right) \, : \, \, X \in P\right\} \nonumber \\ 
    & = \sup_{S \succeq \mathbf{0}, \lambda} \min_{X} \left\{ \langle C, X \rangle - \langle S, X \rangle + \lambda^\top \left( \mathcal{A}_1(X) - b_1 \right) \, : \, \, X \in \Conv (P) \right\}.  \label{eq:sion1} 
\intertext{Since $\mathcal{L}(X, \cdot, \cdot)$ is also linear on $\mathcal{S}^n_+ \times \mathbb{R}^{m_1}$ for all fixed $X \in \Conv(P)$, and $\Conv(P)$ is compact, Sion's minimax theorem~\cite{Sion} implies that we may interchange the order of taking the minimum and the supremum, yielding
    }
  z_{LD}  & = \min_{X} \sup_{S \succeq \mathbf{0}, \lambda} \left\{ \langle C, X \rangle - \langle S, X \rangle + \lambda^\top \left( \mathcal{A}_1(X) - b_1 \right) \, : \, \, X \in \Conv (P) \right\} \nonumber \\
     & = \min_{X}  \left\{ \langle C, X \rangle \, : \, \, X \in \Conv (P),~ \mathcal{A}_1(X) = b_1,~X \succeq \mathbf{0} \right\}, \nonumber
\end{align}
where the final equality follows from the fact that if $\mathcal{A}_1(X) \neq b_1$ or $X \nsucceq \mathbf{0}$, then the inner supremum is unbounded.
\end{proof}

The combination of the results from \Cref{Thm:weakduality} and~\Cref{Thm:AlternativeLD} leads to the following sandwich relation.
\begin{corollary}[Sandwich theorem] \label{Cor:sandwich}
Suppose~\eqref{ISDP} is feasible  and Assumption~\ref{As:bounded} holds. Then, $z_{SDP} \leq z_{LD} \leq z_{MISDP}$.
\end{corollary}
{\begin{remark}  \label{Remark:Slater}
Instead of Assumption~\ref{As:bounded}, we can also establish the sandwich theorem of Corollary~\ref{Cor:sandwich} by assuming that the problem~\eqref{EquivalentLD} has a Slater feasible point and all data matrices are rational (while allowing $P$ to be unbounded).
In that case, $P$ is the intersection of a mixed-integer linear set and a rational polyhedron, hence by Meyer's theorem~\cite{Meyer} $\Conv(P)$ is a polyhedron. Let $\mathcal{A}_3 : \mathcal{S}^n \rightarrow \mathbb{R}^{m_3}$ and $b_3 \in \mathbb{R}^{m_3}$ be the linear operator and vector such that  $\Conv(P) = \{X \in \mathcal{S}^n \, : \,\, \mathcal{A}_3(X) \leq b_3\}$, respectively.
Then,~\eqref{eq:sion1} still holds and can be rewritten to
\begin{align*}
z_{LD} = \sup_{S \succeq \mathbf{0}, \lambda}  -b_1^\top \lambda + \min_{X\in \mathcal{S}^n} \left\{ \langle C - S + \mathcal{A}_1^*(\lambda), X \rangle  \, : \, \,  \mathcal{A}_3(X) \leq b_3 \right\}.
\end{align*}
Observe that the inner minimization problem is linear, so we replace it by its dual. Let $z \geq \mathbf{0}$ denote the dual variable corresponding to $\mathcal{A}_3(X) \leq b_3$, then:
\begin{align}
 & = \sup_{S \succeq \mathbf{0}, \lambda}  -b_1^\top \lambda + \sup_{z \geq \mathbf{0}} \left\{ - b_3^\top z \, : \, \, S = C  + \mathcal{A}_1^*(\lambda) + \mathcal{A}_3^*(z) \right\}  \nonumber\\
 & = \sup_{S \succeq \mathbf{0}, z \geq \mathbf{0}, \lambda} \left\{ -b_1^\top \lambda  - b_3^\top z \, : \, \, S = C  + \mathcal{A}_1^*(\lambda) + \mathcal{A}_3^*(z) \right\}  \nonumber \\
 & = \min_X  \left\{ \langle C, X \rangle \, : \, \, X \in \Conv (P),~ \mathcal{A}_1(X) = b_1,~X \succeq \mathbf{0} \right\}, \nonumber
\end{align}
 where the last equality follows from the latter problem having a Slater feasible point, see e.g., \cite{Ramana1997StrongDF}. Also, observe that the latter minimization problem is attained, due to its feasible set being non-empty and compact. The sequence $z_{SDP} \leq z_{LD} \leq z_{MISDP}$ immediately follows. As boundedness is a more natural property in the MISDPs that are considered in later sections, we stick to Assumption~\ref{As:bounded} in the sequel. \hfill $\qedsymbol$
\end{remark}

 \begin{remark} 
 In the light of Remark~\ref{Remark:Slater}, we cannot relax both Slater feasibility of~\eqref{EquivalentLD} and boundedness of $P$ in order to maintain the result of Corollary~\ref{Cor:sandwich}. For example, consider
    \begin{align*}
        z_{MISDP} = \min \quad & \left\langle \begin{pmatrix}
            -1 & -1 \\
            -1 & 0
        \end{pmatrix}, \begin{pmatrix}
            x_1 & x_2 \\
            x_2 & x_3
        \end{pmatrix} \right \rangle \\
        \text{s.t.} \quad &   \begin{pmatrix}
            x_1 & x_2 \\
            x_2 & x_3
        \end{pmatrix}\succeq 0,~~  x_1 = 0,~~x_2, x_3 \in \mathbb{Z}.
    \end{align*}
    It is clear that $z_{MISDP} = z_{SDP} = 0$. Let  $\mathcal{A}_2(\cdot) = 0$ represent the constraint $x_1 = 0$. Then the set $P$ is unbounded and $\Conv(P)$ does not have a Slater feasible point. The Lagrangian dual function $g(S)$ becomes
    \begin{align*}
        g(S) = \min_{x_2, x_3 {\in \mathbb{Z}}} \left \langle \begin{pmatrix}
            -1 - s_1 & -1 -s_2 \\
            -1 - s_2 & - s_3
        \end{pmatrix}, \begin{pmatrix}
            0 & x_2 \\
            x_2 & x_3
        \end{pmatrix}\right \rangle. 
    \end{align*}
    Since the system $S \succeq \mathbf{0}$, $s_3 = 0$ and $s_2 = -1$ is infeasible, we have $g(S) = -\infty$ for all $S \succeq \mathbf{0}$. Thus, $z_{LD} = -\infty$. \hfill $\qedsymbol$
\end{remark}
}
In the sequel, we briefly describe conditions under which equality throughout the sequence $z_{SDP} \leq z_{LD} \leq z_{MISDP}$ holds. Let $\mathcal{F}_{MISDP}, \mathcal{F}_{SDP}$ and $\mathcal{F}_{LD}$ denote the feasible sets of~\eqref{ISDP}, \eqref{SDP} and \eqref{EquivalentLD}, respectively. Moreover, for any convex set $K \subseteq \mathcal{S}^n$, its normal cone at $X \in K$ is defined as
\begin{align*}
    \mathcal{N}_K(X) := \left\{Z \in \mathcal{S}^n \, : \, \, \langle Z, X \rangle \geq \langle Z, Y \rangle \text{ for all } Y \in K \right\}. 
\end{align*}

\begin{theorem} \label{Thm:StrongDual}
Let $\chi^*_{MISDP}$ and $\chi^*_{LD}$ denote the set of optimizers to \eqref{ISDP} and \eqref{EquivalentLD}, respectively. Then, under Assumption~\ref{As:bounded},
\begin{enumerate}
    \item[(i)] $z_{LD} = z_{MISDP}$ if and only if $-C \in \mathcal{N}_{\mathcal{F}_{LD}}(X^*)$ for all $X^* \in \chi^*_{MISDP}$;
    \item[(ii)] $z_{SDP} = z_{LD}$ if and only if $-C \in \mathcal{N}_{\mathcal{F}_{SDP}}(X^*)$ for all $X^* \in \chi^*_{LD}$.
\end{enumerate}
\end{theorem}
\begin{proof}
$(i)$ Let $X^* \in \chi^*_{MISDP}$. Since $X^* \in \mathcal{F}_{LD}$, $z_{MISDP} = z_{LD}$ if and only if $X^*$ is also an optimizer to~\eqref{EquivalentLD}. The latter holds if and only if $X^*$ is such that $\langle C, X^*\rangle \leq \langle C, Y\rangle$ for all $Y \in \mathcal{F}_{LD}$, which is equivalent to $-C \in \mathcal{N}_{\mathcal{F}_{LD}}(X^*)$.

$(ii)$ The proof of the second statement is very similar, replacing $\chi^*_{MISDP}$ by $\chi^*_{LD}$ and $\mathcal{F}_{LD}$ by $\mathcal{F}_{SDP}$. 
\end{proof}
It follows from \Cref{Thm:StrongDual} that a sufficient condition for $z_{LD} = z_{MISDP}$ when \eqref{ISDP} is bounded is that 
\begin{align*}
    \Conv \left(P \cap \left \{X \in \mathcal{S}^n \, : \, \, \mathcal{A}_1(X) = b_1,~X \succeq \mathbf{0} \right \} \right) = \Conv(P) \cap \left \{X \in \mathcal{S}^n \, : \, \, \mathcal{A}_1(X) = b_1,~X \succeq \mathbf{0} \right \},
\end{align*}
whereas a sufficient condition for $z_{SDP} = z_{LD}$ when \eqref{ISDP} is bounded is
\begin{align*}
    \Conv(P)  = \left \{ X \in \mathcal{S}^n \, : \,\, \mathcal{A}_2(X) = b_2,~ l_{ij} \leq X_{ij} \leq u_{ij} \text{ for all } (i,j) \in \mathcal{J} \right \}. 
\end{align*}

\begin{remark}  \label{Remark:Feasibility} 
    In all of the results in this section, we assumed that~\eqref{ISDP} is feasible. However, when exploiting the Lagrangian dual bounds in a branching framework, there is a need to recognize infeasible subproblems. More precisely, we are interested in obtaining conditions for the Lagrangian dual problem that certify infeasibility of subproblems of the form~\eqref{ISDP}.  

    The problem~\eqref{ISDP} is infeasible if the intersection of the set $P$ and $\{X \, : \, \, \mathcal{A}_1(X) = b_1,~X \succeq 0\}$ is empty. This can happen due to multiple reasons. First, if $P$ is an empty set, the Lagrangian dual function $g(S,\lambda)$ is not defined for any dual point $(S, \lambda)$. This immediately certifies infeasibility of~\eqref{ISDP}, resulting in fathoming the subproblem in a branching framework. If $P$ is non-empty, $g(S,\lambda)$ exists and is finite for all $(S,\lambda)$ where $S \succeq 0$. Therefore, the Lagrangian dual problem~\eqref{Eq:LD} is either feasible or unbounded. 

    Let us first consider the case where~\eqref{Eq:LD} is unbounded. This implies there exists a direction $(S^*, \lambda^*) \in \mathcal{S}^n_+ \times \mathbb{R}^{m_1}$ in which the Lagrangian dual function $g$ keeps on improving. Exploiting the definition of $g$, this happens if and only if there exists a point $X^* \in P$ such that 
    \begin{align} \label{UnboundedConditions}
        S^* - \mathcal{A}_1^*(\lambda^*) \in \mathcal{N}_{\Conv(P)}(X^*) \quad \text{and} \quad \langle \lambda^*, \mathcal{A}_1(X^*) - b_1 \rangle - \langle S, X^* \rangle > 0,
    \end{align} 
    where $\mathcal{N}_{\Conv(P)}(X^*)$ denotes the normal cone of $\Conv(P)$ at the point $X^*$. To verify the validity of these conditions, let $(S^0, \lambda^0)$ be a dual pair for which $X^* \in \text{argmin}_{X \in P} \mathcal{L}(X, S^0, \lambda^0)$. The first condition in~\eqref{UnboundedConditions} then implies that by adding a positive multiple of $(S^*, \lambda^*)$ to $(S^0, \lambda^0)$, the point $X^*$ remains a minimizer of the Lagrangian function at the new point. If the second condition is also satisfied, the value of the Lagrangian function at $X^*$ strictly increases by adding a positive multiple of $(S^*, \lambda^*)$ to $(S^0, \lambda^0)$. Hence, $(S^*, \lambda^*)$ is a ray in the dual feasible set on which the dual function $g$ linearly improves. 

    Finally, when~\eqref{Eq:LD} is feasible, the Lagrangian dual bound $z_{LD}$ is finite, which is exploited in the branching framework. In case the underlying subproblem~\eqref{ISDP} is infeasible, the subproblem is either fathomed since its bound is overruled by an incumbent solution or further branching leads to a detection of infeasibility due to one of the cases described above.
\end{remark}

\section{Hierarchy of Lagrangian dual bounds for ISDP}
\label{sect:Hierarchy of Lagrangian dual}

In this section, we extend on the general theory of Section~\ref{section:Lagrangian duality Theory} and introduce a hierarchy of  Lagrangian dual bounds for purely integer SDPs, i.e., problems of the form \eqref{ISDP} with $\mathcal{J} = [n]\times [n]$.

Without loss of generality, we assume that $B_{ij}$ is of the same form $B \subseteq \mathbb{Z}$ for all $(i,j) \in [n] \times [n]$.
This assumption is non-restrictive, as tighter bounds for certain entries in $X$ can be enforced by the linear constraints.  Let $l$ and $u$ denote the smallest and largest value of $B$, respectively. Hence, we consider problems of the following form:
\begin{align} \label{pureISDP} \tag{$ISDP$}
&\begin{aligned}
z_{ISDP} := \quad \min \quad & \langle C, {X} \rangle  \\
\text{s.t.} \quad &  \mathcal{A}(X) = b,~ X \succeq \mathbf{0}, \quad X \in B^{n \times n}. \end{aligned} 
\end{align}
A main challenge in the construction of the Lagrangian dual problem is to decide on the splitting between $\mathcal{A}_1(X) = b_1$ and $\mathcal{A}_2(x) = b_2$. One choice, which we consider here, is to dualize all of them, i.e., $\mathcal{A}_1(X) = b_1$ captures all linear constraints. To prevent the resulting bound to be equivalent to the continuous SDP relaxation, there is another property of the optimization problem that we can exploit.
In many problems that we consider, i.e., the ones resulting from discrete optimization problems, the feasible matrix variables do not only have integer entries, they also have an upper bound on their rank. This upper bound results from the combinatorial nature of the problem setting. To exploit it, we define for a given integer $1 \leq r \leq n$ the set
\begin{align*}
    \mathcal{S}^n_+(B, r)  := \mathcal{S}^n_+ \cap B^{n \times n} \cap \{X ~:~\rank(X) \leq r\} 
\end{align*}
as the set of PSD matrices having entries in $B$ with an upper bound on its rank. 
We also define $\mathcal{S}^n_+(B) := \mathcal{S}^n_+(B,n)$ as the set of integer PSD matrices without the rank-assumption. Since $B$ is bounded,
the set $\mathcal{S}^n_+(B,r)$ is finite for all $r \leq n$.

For several well-known integer sets $B$, simple descriptions of $\mathcal{S}^n_+(B,r)$ or $\mathcal{S}^n_+(B)$  have been derived, e.g., in the form of a finite generating set of integer rank-1 PSD matrices. Table~\ref{Tab:IntegerPSDSets} lists several of such sets, including an explicit description of its elements and cardinality, see~\cite{BermanXu, LetchfordSorensen,DeMeijerSotirov23}. More results on the structural properties of PSD matrices having integer entries are derived by De Loera et al.~\cite{Loera23}.

\renewcommand\theadfont{\normalsize}
\begin{table}[h]
\centering
\begin{tabular}{@{}lp{0.6\linewidth}l@{}}
\toprule
\thead{Set $\mathcal{D}^n$}                     & \thead{Description} & $|\mathcal{D}^n|$ \\ \midrule
$\mathcal{S}^n_+(\{0,1\})$              
& $X=\sum_{i} x_ix_i^\top$, where $x_i \in \{0,1\}^n$ and each $x_i$ and $x_j$ with $i \neq j$ have non-overlapping support                   
&     $B_{n+1}$         
\\[1.8em]
$\mathcal{S}^n_+(\{0,1\}, 1)$ 
& $X=xx^\top$, where $x \in \{0,1\}^n$          
&     $2^n$              
\\[0.8em]
$\mathcal{S}^n_+(\{0,1\}, r)$ 
& $X=\sum_{i=1}^r x_ix_i^\top$, where $x_i \in \{0,1\}^n$ and each $x_i$ and $x_j$ with $i \neq j$ have non-overlapping support                   
&  $\sum_{k=0}^{r+1} \stirlingii{n+1}{k}$                      
\\[1.8em]
$\mathcal{S}^n_+(\{0, \pm 1\})$          
&  $X=\sum_{i} x_ix_i^\top$, where $x_i \in \{0,\pm1\}^n$ and each $x_i$ and $x_j$ with $i \neq j$ have non-overlapping support                       
&  $D_n$            
\\[1.8em]
$\mathcal{S}^n_+(\{0, \pm 1\}, 1)$          
& $X=xx^\top$, where $x \in \{0,\pm1\}^n$                  
&  $1 + \sum_{k=1}^n \binom{n}{k} 2^{k-1}$  
\\[0.8em]
$\mathcal{S}^n_+(\{0, \pm 1\}, r)$          
& $X=\sum_{i=1}^r x_ix_i^\top$, where $x_i \in \{0,\pm1\}^n$ and each $x_i$ and $x_j$ with $i \neq j$ have non-overlapping support                      
&  $\sum_{k = 0}^n S_B(n,k)$                
\\[1.8em]
$\mathcal{S}^n_+(\{\pm 1\})$            
& $X= xx^\top$, where $x \in \{\pm1\}^n$                 
&    $2^{n-1}$                           
\\ \bottomrule
\end{tabular}
\caption{Descriptions and cardinalities of the sets $\mathcal{S}_+^n(B,r)$ and $\mathcal{S}_+^n(B)$ for different $B$ and $r$. In the third column, $B_n$ denotes the $n$th Bell number, $\stirlingii{n}{k}$ denotes the Stirling number of the second kind, $D_n$ denotes the $n$th Dowling number and $S_B(n,k)$ denotes the $B$-type Stirling number of the second kind. See Appendix~\ref{Appendix:SetPartitions} for details on these numbers and their relation to integer PSD matrices.
\label{Tab:IntegerPSDSets}}
\end{table} 
In the sequel, we let $\mathcal{D}^n = \mathcal{S}^n_+(B,r)$,
where the rank constraint is implied by the constraints $\mathcal{A}(X) = b$ in combination with $X \in B^{n\times n}$ and $X \succeq \bold{0}$.  In other words, $\mathcal{D}^n$ contains the feasible set of~\eqref{pureISDP}.
In this more specified setting, the problem~\eqref{pureISDP} reduces to the following problem:
\begin{align} \label{ISDP_reform} 
&\begin{aligned}
\min \quad & \langle C, {X} \rangle  \\
\text{s.t.} \quad &  \mathcal{A}(X) = b,~X \succeq \bold{0}, ~
 X \in \mathcal{D}^n.  \end{aligned} 
\end{align}
 Although the exact structure of $\mathcal{D}^n$ depends on the problem, it can for example be one of the sets described in Table~\ref{Tab:IntegerPSDSets}. For instance, if $B = \{0,1\}$ and there is no restriction on a rank,
we have $\mathcal{D}^n = \mathcal{S}^n_+(\{0,1\})$.
If we dualize $\mathcal{A}(X) = b$ and $X \succeq \bold{0}$, what remains is $P = \{0,1\}^{n \times n}$. Since $\Conv(P) = [0,1]^{n \times n}$, Theorem~\ref{Thm:AlternativeLD} implies that $z_{LD}$ will be equal to $z_{SDP}$. In what follows, we discuss a general approach that can be applied to strengthen the resulting Lagrangian dual bound.

\medskip 

The set $P$ can be further tightened by keeping (a part of) the constraint $X \succeq \bold{0}$ in the set $P$. If we require that all elements in $P$ must be PSD, the set $P$ becomes $\mathcal{D}^n$, which is intractable to optimize over. However, for relatively small values of $m$, it is possible to optimize over $\mathcal{D}^m$, e.g., by a complete enumeration. Therefore, instead of letting $P$ be the entire set $\mathcal{D}^n$, we can impose a condition on certain submatrices to be $m \times m$ matrices of the form $\mathcal{D}^m$. To that end, assume that $\mathcal{K}^p=\{K_1, \ldots , K_N\}$ denotes a packing on the set $[n]$ where $m_\ell := |K_\ell|$ for  $K_\ell \in \mathcal{K}^p$ is such that $m_\ell \leq p$ for all $K_\ell \in \mathcal{K}^p$, for some given positive integer $p$. We consider the following ISDP that is  equivalent to \eqref{ISDP_reform}:
\begin{align} \label{Def:LDforBSDP}
   \begin{aligned} \min \quad & \langle C, X \rangle \\
    \text{s.t.} \quad & \mathcal{A}(X) = b,~X \succeq \mathbf{0} \\
     & X[K_{\ell}] \in {\mathcal{D}^{m_\ell} \qquad \forall K_\ell \in \mathcal{K}^p} \\
     & X \in B^{n \times n}. \end{aligned}
\end{align} 
Indeed, the equivalence between \eqref{ISDP_reform} and~\eqref{Def:LDforBSDP} follows from the fact that the rank-constraint in the definition of $\mathcal{D}^n$ (if any), is implied by $\mathcal{A}(X) = b$, $X \in B^{n \times n}$ and $X \succeq \mathbf{0}$.
After dualizing $X \succeq \mathbf{0}$ and  $\mathcal{A}(X) = b$, we obtain the following feasible set of remaining constraints, which we denote by $P(\mathcal{K}^p)$: 
\begin{align} \label{Def:PmathcalK}     P(\mathcal{K}^p) := \left\{ X \in \mathcal{S}^n \cap B^{n \times n} \, : \, \, X[K_\ell] \in \mathcal{D}^{m_\ell} \text{ for all } K_\ell \in \mathcal{K}^p\right\}. 
\end{align} 
Let us now check whether this set $P(\mathcal{K}^p)$ can be effectively applied as the feasible set of the Lagrangian dual function $g(S,\lambda)$. Let $X^* \in \arg\min \{\mathcal{L}(X, \hat{S}, \hat{\lambda}) : X \in P(\mathcal{K}^p) \}$ for some $(\hat{S}, \hat{\lambda})$.
First, since the index sets in $\mathcal{K}^p$ are mutually disjoint, the submatrices $\{X[K_\ell] \, : \,  K_\ell \in \mathcal{K}^p \}$ have a pairwise disjoint support in $X$. Hence, the optimal submatrices $X^*[K_\ell]$ can be obtained via a complete enumeration over $\mathcal{D}^{m_\ell}$ for all $K_\ell \in \mathcal{K}^p$ independently.  Moreover, any element $X^*_{ij}$ where $i$ and $j$ do not belong to the same set in $\mathcal{K}^p$ is set to $u$ if $(C - {\hat{S}} + \mathcal{A}_1^*({\hat{\lambda}}))_{ij} < 0$ and to $l$ otherwise. We conclude that the optimization over $P(\mathcal{K}^p)$ is tractable for small $p$.

It follows from Theorem~\ref{Thm:AlternativeLD} that the optimal Lagrangian dual value of~\eqref{Def:LDforBSDP} with dualized constraints $\mathcal{A}(X) = b$ and $X \succeq \mathbf{0}$ equals
\begin{align} \label{Def:LDforBSDPequivalent} 
   \begin{aligned} 
   {z_{LD}^p :=} \min \quad & \langle C, X \rangle \\
    \text{s.t.} \quad & \mathcal{A}(X) = b,~X \succeq \mathbf{0} \\
     & X[K_\ell] \in \mathcal{P}^{m_\ell} \qquad \forall K_\ell \in \mathcal{K}^p \\ 
     & l \leq X_{ij} \leq u~\text{for all } (i,j) \in [n]\times [n], \end{aligned}
\end{align}
where $\mathcal{P}^{m_\ell}  := \Conv (\mathcal{D}^{m_\ell})$ denotes the convex hull of $\mathcal{D}^{m_\ell}$.

The value of $p$, i.e., the maximum size of the sets in $\mathcal{K}^p$, has an impact on the quality of~\eqref{Def:LDforBSDPequivalent}. If $p$ increases, the proportion of entries that is required to be both integer and PSD in the subproblems becomes larger, leading to improved bounds {as provided in the theorem below.} However, this comes at the cost of computation time, as the size of $\mathcal{D}^{m_\ell}$ often grows exponentially, see Table~\ref{Tab:IntegerPSDSets}. Formally, the hierarchy of the Lagrangian dual approach with respect to the value of $p$ can be captured in the following theorem. 

\begin{theorem} \label{THM:HierarchyBounds}
    Let $\mathcal{K}^{p_1}$ and $\mathcal{K}^{p_2}$ with $p_1 \leq p_2$ denote two packings on the set $[n]$, where for each two distinct $i, j \in [n]$ we have: if $i$ and $j$ are in the same subset in $\mathcal{K}^{p_1}$, then $i$ and $j$ are also in the same subset in $\mathcal{K}^{p_2}$. Then,
    \begin{align*}
        z_{SDP} \leq z_{LD}^{p_1} \leq z_{LD}^{p_2} \leq z_{ISDP},
    \end{align*}
    where $z_{LD}^{p_i}$ denotes the Lagrangian dual bound obtained by optimizing over $P(\mathcal{K}^{p_i})$ in each subproblem. 
\end{theorem}
\begin{proof}
    For both packings $\mathcal{K}^{p_i}$, the inequalities $z_{SDP} \leq z_{LD}^{p_i} \leq z_{ISDP}$ follow immediately from Corollary~\ref{Cor:sandwich}. Let $X$ be feasible for~\eqref{Def:LDforBSDPequivalent} with $\mathcal{K}^{p_2}$. For each $K_{\ell} \in \mathcal{K}^{p_1}$, we know that $K_{\ell} \subseteq K_{\ell^*}$ for some $K_{\ell^*} \in \mathcal{K}^{p_2}$. Hence, $X[K_{\ell}]$ is a submatrix of $X[K_{\ell^*}]$, which is an element of $\mathcal{P}^{m_\ell^*}$. Since  the $m_\ell \times m_\ell$ submatrices of all elements in $\mathcal{P}^{m_{\ell^*}}$ are in $\mathcal{P}^{m_\ell}$, it follows that $X[K_{\ell}] \in \mathcal{P}^{m_\ell}$ for all $K_\ell \in \mathcal{K}^{p_1}$. So, $X$ is feasible for~\eqref{Def:LDforBSDPequivalent} with $\mathcal{K}^{p_1}$, resulting in $z_{LD}^{p_1} \leq z_{LD}^{p_2}$. 
\end{proof}

    Semidefinite relaxations including constraints of the form $X[K_\ell] \in \Conv \left( \mathcal{D}^{m_\ell} \right)$ with $\mathcal{D}^{m_\ell} = \mathcal{S}^{m_\ell}_+(\{0,1\})$ or $\mathcal{D}^{m_\ell} = \mathcal{S}^{m_\ell}_+(\{\pm 1\})$  have been studied in~\cite{AdamsAnjosRendlWiegele,RendlGaar}, in which these are referred to as exact subgraph constraints (ESCs). 
    Those papers also consider overlapping submatrices, i.e., $\mathcal{K}^p$ is not required to be a packing, and include the underlying graph structure of the considered problems. Clearly, we can also introduce both ingredients. While overlapping constraints would result in a more difficult  set {$P(\mathcal{K}^p)$} in ~\eqref{def:dual_function} than the one we propose here, additional  graph structure can make the Lagrangian dual problem easier to solve.

\section{Approaches for solving the Lagrangian dual problem}
\label{Sect:SolvingLagDual}
In this section we present several approaches to  compute Lagrangian dual bounds for problems of the form~\eqref{ISDP}. 
Although the optimum to the Lagrangian dual~\eqref{Eq:LD} is theoretically the same as the solution to~\eqref{EquivalentLD}, an explicit description of $\Conv(P)$ is in most cases unavailable. 
Therefore, the here presented algorithms are designed to solve the optimization problem~\eqref{Eq:LD}. Since the Lagrangian dual bounds of the hierarchy presented in Section~\ref{sect:Hierarchy of Lagrangian dual} also fit in the framework of~\eqref{Eq:LD}, the presented algorithms can also be used to compute problems of the form~\eqref{Def:LDforBSDP}. For an overview of algorithms for solving Lagrangian relaxations of MILPs see e.g.,~\cite{Bragin}.

The Lagrangian dual function  $g(S, \lambda)$, see \eqref{def:dual_function}, is a piecewise linear concave function. Due to its non-differentiability, its optimization relies on the use of subgradients. A subgradient of $g(\cdot,\cdot)$ at a point $(S^*, \lambda^*) \in \mathcal{S}^n_+ \times \mathbb{R}^{m_1}$ is a pair $(\Gamma, \gamma) \in \mathcal{S}^n \times \mathbb{R}^{m_1}$ such that
\begin{align*} 
    g(S,\lambda) \leq g(S^*, \lambda^*) + \langle \Gamma, S - S^* \rangle + \gamma^\top(\lambda - \lambda^*) \qquad \text{for all } (S, \lambda) \in \mathcal{S}^n_+ \times \mathbb{R}^{m_1}. 
\end{align*}
Let $X^* \in \arg \min \{\mathcal{L}(X, S^*, \lambda^*) \, : \, \, X \in P\}$. Then, 
\begin{align}\label{gSubgradients}
    g(S,\lambda) \leq g(S^*, \lambda^*) +  \langle -X^*, S - S^* \rangle + (\mathcal{A}_1(X^*) - b_1)^\top (\lambda - \lambda^*).
\end{align}
This shows that $(\Gamma, \gamma) = (-X^*, \mathcal{A}_1(X^*) - b_1)$ is a subgradient of $g(\cdot, \cdot)$ at $(S^*,\lambda^*)$. 
All here presented algorithms exploit subgradients to compute the Lagrangian dual bound~\eqref{Eq:LD}.

\subsection{A projected-deflected  subgradient  algorithm} \label{Subsec:ProjectedSubgradient}

We present a projected-deflected subgradient algorithm for  computing Lagrangian dual bounds. 
The subgradient method was developed by Shor in the 1960s as an
extension of the gradient method to non-differentiable functions, see e.g.,~\cite{Shor:Subgradient}.
The subgradient method is a first-order method in which
the slow rate of convergence is compensated by the  low complexity of each iteration.
The results in \cite{Frangioni2017OnTC} show that  if extensive tuning
is performed, the subgradient method can be competitive with more sophisticated approaches,  when the required tolerance is not too high.

The projected-deflected subgradient algorithm starts with an initial set of dual multipliers $(S^0, \lambda^0)$ $\in \mathcal{S}^n_+ \times \mathbb{R}^{m_1}$. Then, it iteratively obtains $g(S^\ell, \lambda^\ell)$ by minimizing $\mathcal{L}(X, S^\ell, \lambda^\ell)$ over $X \in P $, see \eqref{def:setP}, yielding an optimal solution $X^\ell$.
As indicated before, we assume this optimization to be tractable, e.g., for the problem~\eqref{Def:LDforBSDP} we exploit the structure of the sets provided in Section~\ref{sect:Hierarchy of Lagrangian dual}.
The algorithm now computes the subgradient $(\Gamma^\ell, \gamma^\ell) := (-X^\ell, \mathcal{A}_1(X^\ell) - b_1)$. The dual multipliers are updated by a step update in the direction $(D^\ell, d^\ell) \in \mathcal{S}^n \times \mathbb{R}^{m_1}$, where this direction is based on the subgradient  $(\Gamma^\ell, \gamma^\ell)$. This can be done in several ways, as we discuss below.
Then, we set
\begin{align} \label{Eq:DualUpdateSG}
  S^{\ell+1} \gets \mathcal{P}_{\mathcal{S}^n_+}(S^\ell + \alpha^\ell D^\ell) \quad \text{and} \quad \lambda^{\ell+1} \gets \lambda^{\ell} + \beta^\ell d^\ell,  
\end{align}
where $\alpha^\ell$ and $\beta^\ell$ are appropriate stepsize parameters. The resulting dual matrix $S^\ell + \alpha^\ell D^\ell$ is projected onto the PSD cone in order to stay dual feasible. 

Different strategies for the choice of the dual updates $(D^\ell, d^\ell)$ and the stepsize parameters $\alpha^\ell$ and $\beta^\ell$ are proposed in the literature for the case of integer linear programming. For an overview of such strategies, see e.g.,~\cite{Frangioni2017OnTC} and references therein.  
In our preliminary study \cite{deMijerThesis}, we considered several of those strategies and altered them to the case of MISDP. In particular,  the standard subgradient update with Polyak’s stepsize~\cite{Polyak}, the deflected subgradient algorithm using the update scheme from Camerini et al.~\cite{CameriniEtAl}, the deflected subgradient algorithm using the update scheme from Sherali and Ulular~\cite{SheraliUlular}, and the conditional subgradient algorithm with the Polyak update~\cite{Polyak} were implemented.
In these  preliminary tests the deflected subgradient algorithm with the update scheme from~\cite{SheraliUlular} 
turned out to perform best on the subset of instances tested. Therefore, we further consider only that variant of the subgradient algorithm.
\medskip

The deflected subgradient algorithm \cite{SheraliUlular} constructs the dual update vector as a linear combination between the subgradient vector and the previous dual update vector, i.e.,
    \begin{align}\label{dualUpdate1}
        D^\ell = \Gamma^\ell + {\varphi^\ell_1} D^{\ell-1} \quad \text{and} \quad d^\ell = \gamma^\ell + {\varphi_2^\ell} d^{\ell-1},
    \end{align}
    where { $\varphi_1^\ell, \varphi_2^\ell \geq 0$} are deflection parameters. This choice of the dual update intends to reduce a zigzagging pattern between consecutive dual multipliers.
    Sherali and Ulular~\cite{SheraliUlular} propose to take
        \begin{align}\label{dualUpdate2} 
            \varphi^\ell_1 = \frac{||\Gamma^\ell||_F}{||D^\ell||_F} \quad \text{and} \quad \varphi^\ell_2 = \frac{||\gamma^\ell||}{||d^\ell||},
        \end{align}
        in order to let the new dual update vector bisect the angle between the current subgradient $\Gamma^\ell$ (resp.~$\gamma^\ell$) and the previous dual update $D^{\ell-1}$ (resp.~$d^{\ell-1}$).  
        
    The stepsize parameters $\alpha^\ell$ and $\beta^\ell$  can be chosen according to the Polyak update~\cite{Polyak}, that is
    \begin{align}\label{stepsize}
        \alpha^\ell := \frac{\mu^\ell_1 (U^* - g(S^\ell, \lambda^\ell))}{||\Gamma^\ell ||^2_F} \quad \text{and} \quad \beta^\ell := \frac{\mu^\ell_2 (U^* - g(S^\ell, \lambda^\ell))}{||\gamma^\ell ||^2}. 
    \end{align}
    Here $U^*$ is an upper bound on the optimal Lagrangian dual value $z_{LD}$, which can be obtained by a heuristic on the primal problem, and $0 < \mu_1^\ell,  \mu_2^\ell \leq 2$ are stepsize parameters. The justification for this step length is its theoretical convergence when $U^*$ is set equal to $z_{LD}$, see~\cite{Polyak}.

A pseudo-code of the projected-deflected subgradient algorithm is given in Algorithm~\ref{Alg:projectedSubgradient}. 

\begin{algorithm}[h!]
\footnotesize
\caption{Projected-deflected subgradient algorithm for solving~\eqref{Eq:LD}}\label{Alg:projectedSubgradient}
\begin{algorithmic}[1]
\Require $C, \mathcal{A}_1, \mathcal{A}_2, b_1, b_2,$ and $B_{ij}$ for all $(i,j) \in \mathcal{J}$.
\State Initialize dual pair $(S^0, \lambda^0) \in \mathcal{S}^n_+ \times \mathbb{R}^{m_1}$ and set $\ell = 0$, $\hat{z}=- \infty$. 
\While{stopping criteria are not met}
\State Compute $g(S^\ell, \lambda^\ell)$ and obtain $X^\ell {\in} \arg \min \{\mathcal{L}(X, S^\ell, \lambda^\ell) \, : \, \, X \in P\}$.
\If {$\hat{z} < g(S^\ell, \lambda^\ell)$} 
\State $\hat{z} \gets g(S^\ell, \lambda^\ell)$
\EndIf
\State Compute subgradients $\Gamma^\ell = - X^\ell$, $\gamma^\ell = \mathcal{A}_1(X^\ell) - b_1$. 
\State Update stepsize parameters $\alpha^\ell$, $\beta^\ell$ using \eqref{stepsize},
and dual updates $D^\ell, d^\ell$ using \eqref{dualUpdate1}--\eqref{dualUpdate2}. 
\State Update $S^{\ell+1} \gets \mathcal{P}_{\mathcal{S}^n_+}(S^\ell + \alpha^\ell D^\ell)$, $\lambda^{\ell+1} \gets \lambda^\ell + \beta^\ell d^\ell$. 
\State $\ell \gets \ell+1$
\EndWhile

\Ensure $\hat{z}$
\end{algorithmic}
\end{algorithm}
The subgradient algorithm stops  when the difference between consecutive dual multipliers is small, i.e., $||S^\ell - S^{\ell-1}||_F < \varepsilon_1$ and $||\lambda^\ell - \lambda^{\ell-1}|| < \varepsilon_2$ for some predefined parameters $\varepsilon_1, \varepsilon_2 > 0$. Moreover, we implement a stagnation criterion:  the algorithm stops if there has been no improvement in the value of the objective function for the last $N_{\text{stag}}$ iterations.

\subsection{A projected-accelerated subgradient algorithm} \label{Subsec:Nesterov}

We consider here a version of the accelerated gradient method introduced by Nesterov~\cite{Nesterov1983AMF}. 
Namely, we adopt the approach from~\cite{Nesterov1983AMF} and adjust it for the case of MISDP.  
The resulting projected-accelerated subgradient algorithm  computes the  Lagrangian dual bound \eqref{Eq:LD}. 

The projected-accelerated subgradient algorithm  starts with an initial set of dual multipliers $(S^0, \lambda^0) \in \mathcal{S}^n_+ \times \mathbb{R}^{m_1}$ and introduces auxiliary variables  
$Y^0:=S^0$ and $\delta^0 := \lambda^0$. 
Then, the algorithm computes 
$g(Y^0, \delta^0)=\min_{X \in P} \mathcal{L}(X, Y^0, \delta^0)$, 
yielding an optimal solution $X^0$ and corresponding subgradient $(\Gamma^0,\gamma^0)$. 
Next, the algorithm iteratively performs updates of dual multipliers by employing the following two-step iteration 
\begin{align*}  
S^{\ell+1} \gets \mathcal{P}_{\mathcal{S}^n_+}(Y^\ell + \alpha^\ell\Gamma^\ell), \quad
& Y^{\ell+1} = S^{\ell+1} + {\frac{\eta_\ell-1}{\eta_{\ell+1}}} (S^{\ell+1}-S^{\ell}) \\
\lambda^{\ell+1} \gets \delta^{\ell} + \beta^\ell \gamma^\ell, \quad
& \delta^{\ell+1} = \lambda^{\ell+1} + {\frac{\eta_\ell-1}{\eta_{\ell+1}} }(\lambda^{\ell+1}-\lambda^{\ell}),  \quad \ell=0,1,\ldots,
\end{align*}
where $\alpha^\ell$, $\beta^\ell$ are  stepsize parameters  chosen with respect to the Polyak update~\cite{Polyak}, see also \eqref{stepsize},
  and {$\{ \eta_\ell\}_{\ell=0}^\infty$}  the sequence generated as follows
$\eta_0=0$ and $\eta_{\ell+1} =(1+\sqrt{1+4\eta_\ell^2})/2$, $\ell=0,1,\ldots.$ 
 Iteratively, the algorithm performs a step of gradient ascent to go from $Y^\ell$  (resp.~$\delta^\ell$) to $S^{\ell+1}$ (resp.~$\lambda^{\ell+1}$), followed by a correction in the auxiliary variables, i.e., the so-called look-ahead step.
 Nesterov's two-point step iteration approach from~\cite{Nesterov1983AMF} results in an accelerated gradient algorithm that achieves an improved convergence rate with respect to a simple gradient algorithm.
 Matrix  $(Y^\ell + \alpha_\ell\Gamma^\ell)$  is projected onto the PSD cone so that the dual variable stays dual feasible.
 A pseudo-code of projected-accelerated subgradient algorithm is presented in Algorithm~\ref{Alg:projectedAcceleratedSubgradient}.

We use the same stopping criteria for this algorithm as for Algorithm~\ref{Alg:projectedSubgradient}, where we monitor a stagnation criterion with respect to $Y^\ell$ instead of $S^\ell$.
\begin{algorithm}[h!]
\footnotesize
\caption{Projected-accelerated subgradient algorithm for solving~\eqref{Eq:LD}}\label{Alg:projectedAcceleratedSubgradient}
\begin{algorithmic}[1]
\Require $C, \mathcal{A}_1, \mathcal{A}_2, b_1, b_2$, and $B_{ij}$ for all $(i,j) \in \mathcal{J}$
\State Initialize dual pair $(S^0, \lambda^0) \in \mathcal{S}^n_+ \times \mathbb{R}^{m_1}$.
\State Set $Y^0=S^0$, $\delta^0=\lambda^0$,  ${\eta_0=0}$ and $\ell = 0$.
\While{stopping criteria are not met}
\State Obtain $X^\ell {\in} \arg \min \{\mathcal{L}(X, Y^\ell, \delta^{\ell}) \, : \, \, X \in P\}$.
\State Compute subgradients $\Gamma^\ell = - X^\ell$, $\gamma^\ell = \mathcal{A}_1(X^\ell) - b_1$. 
\State Update stepsize parameters $\alpha^\ell$, $\beta^\ell$ using \eqref{stepsize}.
\State Update $S^{\ell+1} \gets \mathcal{P}_{\mathcal{S}^n_+}(Y^\ell +\alpha^\ell\Gamma^\ell)$,~~
  $\lambda^{\ell+1} \gets \delta^{\ell} + \beta^\ell \gamma^\ell$.
\State Compute {$\eta_{\ell+1} = \frac{1}{2}\left( 1+\sqrt{1+4\eta_\ell^2} \right)$.}
\State Update  $Y^{\ell+1} \gets  S^{\ell+1} + {\frac{\eta_\ell-1}{\eta_{\ell+1}}} (S^{\ell+1}-S^{\ell})$, ~~
$\delta^{\ell+1} \gets \lambda^{\ell+1} + {\frac{\eta_\ell-1}{\eta_{\ell+1}} }(\lambda^{\ell+1}-\lambda^{\ell})$. 
\State $\ell \gets \ell+1$
\EndWhile
\State Compute $\hat{z}=g(S^{\ell}, \lambda^{\ell})$.
\Ensure $\hat{z}$
\end{algorithmic}
\end{algorithm}

\subsection{A projected bundle algorithm} \label{Subsec:Projectedbundle}

We present a projected bundle algorithm for solving the Lagrangian dual problem \eqref{Eq:LD}.
The bundle method was introduced in the 1970's \cite{Kiwiel,Lemarchal1978NonsmoothOA} as a method for optimizing non-smooth functions.
The bundle method has already been exploited for computing SDP bounds, however, our algorithm differs from  SDP-based bundle algorithms from the literature. Namely,  function evaluations in our algorithm are over a polyhedral set, while in other bundle algorithms function evaluations are over a spectrahedron. We consider here a version of the proximal bundle algorithm~\cite{Kiwiel1990ProximityCI}.

The bundle algorithm  starts with an initial set of dual multipliers $(S^0, \lambda^0) \in \mathcal{S}^n_+ \times \mathbb{R}^{m_1}$ and obtains $X^0$ from  $g(S^0, \lambda^0)=\min \left\{ \mathcal{L}(X, S^0, \lambda^0) \, : \, \, X \in P \right\}$,  see~\eqref{def:dual_function}. 
Then, the algorithm computes a subgradient of $g(\cdot, \cdot)$ at $(S^0,\lambda^0)$, that is,  $(\Gamma^0, \gamma^0) = (-X^0, \mathcal{A}_1(X^0) - b_1)$.
The bundle algorithm is an iterative algorithm that maintains in each iteration the  best current approximation $(\widehat{S}, \widehat{\lambda})$ to the maximizer of $g(\cdot,\cdot)$,  (some of) previously computed points $X^i\in P$ and  the corresponding subgradients
$(\Gamma^i, \gamma^i) = (-X^i, \mathcal{A}_1(X^i) - b_1)$. 
The set that contains feasible points $X^i$ and corresponding subgradients is  called {\em the  bundle} and denoted by $\mathcal{B}$. We denote by $J_\mathcal{B}$ its index set.

One iteration of the algorithm is as follows. We assume to have the bundle ${\mathcal B}= \{( X^i,\Gamma^i, \gamma^i) :~~i \in J_\mathcal{B}\}$.
To compute the next trial point $(S_{\rm trial}, \lambda_{\rm trial})$, the bundle method combines the following two concepts. 
The first concept is to  approximate the function $g(S,\lambda)$ in the neighborhood of the previous iterates. That is accomplished by the following subgradient model:
$$
g_{\rm appr}(S,\lambda) := \min \left\{ \mathcal{L}(X, S, \lambda) \, : \, \, X \in \mbox{conv}\left ( \{X^i \, : \, \, i \in J_\mathcal{B} \}\right ) \right\}, 
$$
which can be rewritten as follows
\begin{align} \label{gapprox}
g_{\rm appr}(S,\lambda)
& =  \min_{\sigma \in \Delta} ~  \sum_{i \in J_\mathcal{B}} \sigma_i \left(  \langle  C, X^i \rangle  +  \langle  \Gamma^i, S \rangle  + \lambda^\top \gamma^i  \right),
\end{align} 
where  $\Delta := \{ \sigma \in \mathbb{R}^{|{J_\mathcal{B}}|} :~ \mathbf{1}^\top \sigma =1, ~ \sigma \geq 0 \}$.
The second concept is to stay in the vicinity of the center point  $(\widehat{S}, \widehat{\lambda})$.
This can be achieved by using the proximal point approach that exploits the Moreau–Yosida regularization of $g_{\rm appr}(\cdot,\cdot)$. The combination of the two ideas results in the following concave problem whose solution gives  a new trial point:
\begin{align} 
\begin{aligned} \label{maxBundle}
 \quad \max \quad & g_{\rm appr}(S,\lambda) - \frac{t}{2} ( || S - \widehat{S} ||_F^2 + ||\lambda - \widehat{\lambda} ||^2)    \\
\text{s.t.} \quad &  S \succeq \mathbf{0},~ \lambda \in \mathbb{R}^{m_1}. \end{aligned} 
\end{align} 
Here $t> 0$ is known as the proximity parameter.
Since the above problem is a concave quadratic optimization problem, the complexity of each bundle iteration is more costly than a subgradient iteration.
Moreover, the above problem is a quadratic semidefinite programming problem. 
Nevertheless, we proceed in the line of the bundle literature that commonly proposes to solve the dual problem of \eqref{maxBundle}.
That is, we aim to solve the following optimization problem:
\begin{align}
\min_{Q\succeq 0}  ~& \max_{\substack{S\in {\mathcal S}^n \\ \lambda \in \mathbb{R}^{m_1}}} \min_{\sigma \in \Delta}~   \sum_{i \in J_\mathcal{B}} \sigma_i \left(  \langle  C, X^i \rangle  +  \langle  \Gamma^i, S \rangle  + \lambda^\top \gamma^i  \right)
- \frac{t}{2} ( || S - \widehat{S} ||_F^2 + ||\lambda - \widehat{\lambda} ||^2)  + \langle S,Q \rangle \nonumber \\
 =& \min_{\substack{\sigma\in \Delta\\Q\succeq 0}}  ~\max_{\substack{S\in {\mathcal S}^n \\ \lambda \in \mathbb{R}^{m_1}}}   \sum_{i \in J_\mathcal{B}} \sigma_i \left(  \langle  C, X^i \rangle  +  \langle  \Gamma^i, S \rangle  + \lambda^\top \gamma^i  \right)
- \frac{t}{2} ( || S - \widehat{S} ||_F^2 + ||\lambda - \widehat{\lambda} ||^2)  + \langle S,Q \rangle. \label{MinMaxBundle}
\end{align} 
Indeed, we may swap around the minimization w.r.t.~$\sigma$ and the maximization w.r.t.~$(S, \lambda)$, due to the set $\Delta$ being compact and the objective function being strongly concave in $(S,\lambda)$ and linear in $\sigma$.
From the first-order optimality conditions for the inner maximization problem we have
\begin{align*}
S =  \widehat{S} + \frac{1}{t} \left (Q+ \sum_{i \in J_\mathcal{B}} \sigma_i \Gamma^i \right ),
\quad 
\lambda = \widehat{\lambda} + \frac{1}{t} \sum_{i\in J_\mathcal{B}} \sigma_i \gamma^i.
\end{align*}
We incorporate this into \eqref{MinMaxBundle} and obtain the following minimization problem:
\begin{align}
\begin{aligned} \label{finalBundleSolve}
\min_{\sigma\in \Delta,Q\succeq 0} & ~ \frac{1}{2t} \langle Q,Q \rangle  
+ \frac{1}{2t}  \left \lVert  \sum_{i=1}^\ell  \sigma_i \gamma^i \right \rVert^2
+ \frac{1}{2t}  \left \langle  \sum_{i=1}^\ell  \sigma_i X^i,   \sum_{i=1}^\ell  \sigma_i X^i \right \rangle \\
& +  \left \langle  \sum_{i=1}^\ell  \sigma_i X^i, C -  \widehat{S}  - \frac{1}{t} Q  \right \rangle 
+  \left(  \sum_{i=1}^\ell   \sigma_i \gamma^i \right)^\top  \widehat{\lambda}   
+  \left \langle \widehat{S},Q  \right \rangle.
\end{aligned}
\end{align}
This problem is a convex quadratic optimization problem, and we solve it approximately by keeping one set of the variables constant.
In particular, keeping $\sigma$ fixed we solve for $Q = \mathcal{P}_{\mathcal{S}^n_+} \left ( \sum_{i \in J_\mathcal{B}}  \sigma_i X^i  - t \widehat{S} \right ),$
and keeping $Q$ constant results in a convex quadratic problem over the set $\Delta$ that can be efficiently solved.
Thus, we start with $Q=\mathbf{0}$ and solve for $\sigma$, which we thereafter keep constant to solve for $Q$,
and iterate this process several times to get (approximate) solutions to \eqref{finalBundleSolve}.
In practice, one  has to make a few iterations to obtain a  good approximate solution. 
However, the most expensive computation is projection onto the PSD cone.
Using the final estimates of $Q$ and $\sigma$, say  $\widetilde{Q}$ and $\widetilde{\sigma}$,  we compute the new trial point as follows
\begin{align}\label{newSLambda}
S_{\rm trial} =   \mathcal{P}_{\mathcal{S}^n_+} \left (\widehat{S} + \frac{1}{t} \left (\widetilde{Q}+ \sum_{i \in J_\mathcal{B}} \widetilde{\sigma}_i \Gamma^i \right )\right ) \quad \text{and}
\quad
\lambda_{\rm trial} = \widehat{\lambda} + \frac{1}{t} \sum_{i \in J_\mathcal{B}} \widetilde{\sigma}_i \gamma^i.
\end{align}
To finish one iteration of the bundle algorithm, we still need to evaluate \eqref{def:dual_function} at $(S_{\rm trial},\lambda_{\rm trial})$,
which results in the matrix $X_{\rm trial} \in P$.
If $\widetilde{\sigma}_i = 0$ for some $i$, the corresponding subgradients have no influence in the optimization and therefore we
remove them from the bundle~${\mathcal B}$ and from further  computations.

The bundle algorithm now needs to decide whether or not $(S_{\rm trial}, \lambda_{\rm trial})$ becomes the next center point $(\widehat{S}, \widehat{\lambda})$.
The algorithm first checks whether  the improvement of the objective function in the trial point is at least a  $\rho \in (0,1)$ fraction of the improvement  that the subgradient model provides.
If it does, $(S_{\rm trial}, \lambda_{\rm trial})$ becomes the next center point, and the bundle  $\mathcal B$ is updated with  $X_{\rm trial}$ and the corresponding subgradient.
In the literature this is called  a serious step.
Otherwise, the algorithm keeps the last center point, but nevertheless updates  $\mathcal{B}$ with  $X_{\rm trial}$ and the corresponding subgradient, 
resulting into a so-called  null step.

For updating the  proximity parameter $t$, we use an  update similar to the one proposed by Kiwiel~\cite{Kiwiel1990ProximityCI}.
The projected bundle algorithm stops if the difference of the function value at the current center point and the function value of the subgradient model at the trial point is smaller than some prescribed tolerance. The algorithm also stops if it reaches the maximum number of iterations. 
A pseudo-code of the projected bundle algorithm is given in Algorithm~\ref{Alg:projectedBundle}.
\begin{algorithm}[h!]
\footnotesize
\caption{Projected bundle algorithm for solving~\eqref{Eq:LD}}\label{Alg:projectedBundle}
\begin{algorithmic}[1]
\Require $C, \mathcal{A}_1, \mathcal{A}_2, b_1, b_2$, and  $B_{ij}$ for all $(i,j) \in \mathcal{J}$
\State Initialize dual pair $(S^0, \lambda^0) \in \mathcal{S}^n_+ \times \mathbb{R}^{m_1}$, $t$ and $\rho$, and set $\ell = 0$.
\State Compute $X^0 {\in} \arg \min \left\{ \mathcal{L}(X, S^0, \lambda^0) \, : \, \, X \in P \right\}$ and the subgradient $(\Gamma^0,\gamma^0)$.
Set the bundle to ${\mathcal B}=\{  (X^0,\Gamma^0, \gamma^0) \}$.
\While{stopping criteria are not met}
\State Solve \eqref{MinMaxBundle} to obtain  $\widetilde{Q}$ and $\widetilde{\sigma}$. 
\State  Determine next trial point $(S_{\rm trial},\lambda_{\rm trial})$ using \eqref{newSLambda}.
\State Compute $g(S_{\rm trial},\lambda_{\rm trial})$ and obtain $X_{\rm trial} {\in} \arg \min \left\{ \mathcal{L}(X, S_{\rm trial}, \lambda_{\rm trial}) \, : \, \, X \in P \right\}$.
\State Obtain subgradients $\Gamma_{\rm trial}=-X_{\rm trial}$, $\gamma_{\rm trial} = {\mathcal A}_1(X_{\rm trial})-b_1$.
\State Decide whether $(S_{\rm trial}, \lambda_{\rm trial})$ becomes $(\widehat{S}, \widehat{\lambda})$ (serious step) or not (null step).
\State Update the bundle $\mathcal{B}$ and the parameter $t$, $\ell \gets \ell +1$.
\EndWhile

\Ensure $g(\widehat{S}, \widehat{\lambda})$
\end{algorithmic}
\end{algorithm}

\section{The max-{\em k}-cut problem} \label{Sect:MaxKCut}

The max-$k$-cut problem is the problem of partitioning the vertex set of a graph into  $k$ subsets such that the total weight
of edges joining different sets is maximized. For the case that $k=2$, the max-$k$-cut problem is known as the max-cut problem. The max-$k$-cut problem is $\mathcal{NP}$-hard~\cite{Arora1992ProofVA} 
and has many applications such as VLSI design, frequency planning, statistical physics and sports scheduling~\cite{BarahonaEtAl,Mitchell}.

Let $G = (V,E)$ be an undirected graph with $n$ vertices and  $W = (w_{ij})$ a weight matrix   such that $w_{ij} = 0$ if $\{i,j\} \notin E$.
Let $Z \in \{0,1\}^{n\times k}$ denote the characteristic matrix of a partitioning of $V$, where $Z_{ij} = 1$ if vertex $i$ is in subset $j$ and $Z_{ij} = 0$ otherwise. Then, $Y := ZZ^\top$ is a binary $n \times n$ PSD matrix of rank at most $k$ that satisfies $\diag(Y) = \bold{1}_n$.  
Following~\cite[Corollary 2]{DeMeijerSotirov23}, the rank constraint on the binary matrix $Y$ can be enforced by the linear matrix inequality $\begin{psmallmatrix}
     k &  \mathbf{1}_n^\top  \\
     \mathbf{1}_n & Y 
    \end{psmallmatrix} \succeq \mathbf{0}$, leading to the following ISDP formulation where $L := \Diag(W\mathbf{1}_n) - W$ denotes the weighted Laplacian matrix of $G$: 
\begin{align} \label{MaxKCutISDP}
\begin{aligned} \max \quad & \frac{1}{2} \langle L, Y \rangle\\ 
    \text{s.t.} \quad & \diag(Y) = \mathbf{1}_n, \quad  
    \begin{pmatrix}
     k &  \mathbf{1}_n^\top  \\
     \mathbf{1}_n & Y 
    \end{pmatrix} \succeq \mathbf{0}, \quad Y \in{\mathcal S}^n, \quad Y\in \{0,1\}^{n\times n}.
         \end{aligned} 
        \end{align}
A more compact reformulation of the max-$k$-cut problem can be obtained by rewriting the linear matrix inequality of~\eqref{MaxKCutISDP} to $Y - \frac{1}{k}\bold{J} \succeq \bold{0}$ using the Schur complement  lemma. By defining the new variable $X := \frac{k}{k-1}(Y - \frac{1}{k}\bold{J})$, the formulation~\eqref{MaxKCutISDP} is equivalent to the following discrete program:
\begin{align} \label{Prob:MaxCut}
    \begin{aligned} \max \quad &  \frac{k-1}{2k} \langle L, X \rangle \\ 
    \text{s.t.} \quad & \diag(X) = \mathbf{1}_n, \quad  X \succeq \mathbf{0}, \quad 
         X  \in \left \{ \frac{-1}{k-1},1 \right \}^{n \times n}.
        \end{aligned} \qquad \qquad \quad
\end{align}
Here, $X_{ij}= -1/(k-1)$ if vertices $i,j$  are in  different subsets of the partition and $X_{ij}=1$ otherwise. One can easily show that a matrix $X$ that is feasible for~\eqref{Prob:MaxCut} is of rank at most $k-1$. The formulation~\eqref{Prob:MaxCut} is well-known in the literature, see~\cite{Frieze1995ImprovedAA}, where also a geometric meaning of the program is provided. 

The following well-known SDP relaxation for the max-$k$-cut problem is obtained from \eqref{Prob:MaxCut}
 by relaxing $X_{ij}  \in \left \{ \frac{-1}{k-1},1 \right \}$ to $ \frac{-1}{k-1} \leq X_{ij} \leq 1$: 
\begin{align}
    \begin{aligned} \max \quad & \frac{k-1}{2k}  \langle L, X \rangle \\ 
    \text{s.t.} \quad & \diag(X) = \mathbf{1}_n, \quad X \succeq \mathbf{0},\quad X_{ij} \geq  \frac{-1}{k-1}, \quad \forall i,j\in [n].
          \end{aligned} \qquad \quad \label{MaxKCutSDP}
\end{align}
The constraints   $X_{ij}\leq 1$ for all $i,j$ are omitted since they are  redundant due to  $X\succeq \mathbf{0}$ and $\diag(X) = \mathbf{1}_n$.
For $k=2$, also the constraint $X \geq -\mathbf{J}_n $ is redundant, and the resulting problem is equivalent to the basic SDP relaxation for the max-cut problem, see e.g.,~\cite{GoemansWilliamson}.

The SDP relaxation can be further tightened by adding valid inequalities such as the triangle and the clique inequalities, see e.g.,~\cite{Ghaddar2011ABA}. The triangle inequalities ensure that if any two vertices $i$ and $j$ are in the same partition, and vertices $j$ and $h$ are in the same partition, then also vertices $i$ and $h$ have to be in the same partition. 
The resulting $3\binom{n}{3}$ triangle inequalities are of the form:
\begin{align} \label{triangleIneq}
    X_{ij} + X_{jh} - X_{ih} \leq 1, \quad i,j,h\in V.
\end{align}
The triangle inequalities are exploited in \cite{Fischer2006ComputationalEW,RendlEtAlMaxCut} for computing strong bounds for the max-cut problem.
The clique inequalities ensure that for each subset of $k+1$ vertices, at least two of the vertices belong to the same partition. Those inequalities are of the form:
\begin{align} \label{cliqueIneq}
   \sum_{i,j \in S,~i<j} X_{ij} \geq -\frac{k}{2}, \quad \forall S\subseteq V \quad \mbox{where} \quad |S|=k+1.
\end{align}
There are $\binom{n}{k+1}$ clique inequalities.
While  separation of triangle inequalities can be done  efficiently, the exact separation of clique inequalities is $\mathcal{NP}$-hard in general \cite{ChopraRao}. 
There exist other types of inequalities that one may add to the SDP relaxation~\eqref{MaxKCutSDP}, e.g., general clique inequalities or wheel inequalities~\cite{SousaAnjos}. 

In the sequel we follow the idea of Section~\ref{sect:Hierarchy of Lagrangian dual} to derive  bounds for the max-$k$-cut problem.
We first consider the following discrete SDP that is equivalent to~\eqref{Prob:MaxCut}:
\begin{align}  \label{Prob:MaxKCutISDP}
\begin{aligned} 
\max \quad & \frac{k-1}{2k} \left\langle L, X \right\rangle \\ 
    \text{s.t.} \quad & \diag(X) = \mathbf{1}_n, \quad  X \succeq \mathbf{0}, \quad  X  \in \left \{ \frac{-1}{k-1},1 \right \}^{n \times n}, \quad X[K_\ell] \in \widehat{\mathcal{D}}^{m_\ell} \qquad \forall K_\ell \in \mathcal{K}^p \\
    \end{aligned}
    \end{align} 
 where 
\begin{align}
    \widehat{\mathcal{D}}^{m_\ell} &:= 
     \left\{ X \in {\mathcal S}^{m_\ell}_+ \cap \left\{\frac{-1}{k-1},1 \right \}^{m_\ell \times m_\ell} \, : \, \, \diag(X)= \mathbf{1}_{m_\ell},~  \rank(X)\leq k-1 \right \},
    \label{Def:setDhat}
\end{align} 
and $p$ is a given integer such that $3\leq p \leq n$.
The set $\widehat{\mathcal{D}}^{m_\ell}$ is the image of the elements in $\mathcal{S}^{m_\ell}_+(\{0,1\},k)$ with diagonal elements equal to one under the mapping $Y \rightarrow \frac{k}{k-1}(Y - \frac{1}{k}\bold{J})$. Therefore, we can fully enumerate over its elements by exploiting Table~\ref{Tab:IntegerPSDSets}.  

The collection $\mathcal{K}^p$
is a collection of subsets of $[n]$, each of size at most $p$.
Observe that the exact submatrix constraint $X[K_\ell]$ for some set $K_\ell$ only concerns the off-diagonal elements of $X[K_\ell]$. Consequently, the submatrix constraints for two sets $K_1$ and $K_2$ that intersect on a single index, i.e., $|K_1 \cap K_2| = 1$, can be evaluated independently. Thus, we no longer require $\mathcal{K}$ to consist of mutually disjoint sets, it is sufficient to require that the elements in $\mathcal{K}$ do not pairwise intersect in more than one index. We call a collection of subsets satisfying this property an \emph{edge-packing} of~$[n]$. 
We construct the Lagrangian dual by dualizing the constraints $X \succeq \mathbf{0}$, which results in  the Lagrangian dual bound, {$z_{LD}^p$}.

\section{Computational results for the max-$k$-cut problem} \label{Subsec:MaxCutNumerics}

In this section we provide a computational study on the strength of the Lagrangian dual bounds obtained from the discrete SDP formulation of the max-$k$-cut problem discussed in Section~\ref{Sect:MaxKCut}.

\subsection{Design of the computational experiments} \label{Subsec:design}
In this section we present the design of our computational experiments. More specifically, we present the considered instances, the various bounding approaches that we take into account and several implementation details of the algorithms for computing Lagrangian dual bounds.

\subsubsection{Description of instances}
\label{Subsec:DataDescription}

We test our approaches on several instance classes from the literature\footnote{Data instances are publicly available at the BiqMac library (\url{https://biqmac.aau.at/biqmaclib.html}) and the BiqBin library (\url{https://www.biqbin.eu/}.)}. These classes are as follows: 
\begin{itemize}
    \item \textbf{Rudy instances}: These instances are randomly generated using the machine-independent graph generator rudy~\cite{rudygenerator}. The instances `g05\_$n$' are unweighted graphs on $n$ vertices, where each edge is included with probability ${1}/{2}$. The instance classes `pm1d\_$n$' and `pm1s\_$n$' contain graphs on $n$ vertices with edge densities $0.9$ and $0.1$, respectively, having edge weights chosen uniformly at random from $\{0,\pm 1\}$. The graphs `pw$d$\_$n$' are defined on $n$ vertices with edge density $d \in \{0.1, 0.5, 0.9\}$, where the weights are integers from $\{0, \ldots, 10\}$ chosen uniformly at random. Finally, the class `w$d$\_$n$' is defined similarly, except for the weights being chosen as integers from $\{-10, \ldots, 10\}$. Each instance class consists of 10 randomly generated instances. For  $n\in \{60, 80, 100\}$, the resulting test set consists of 130 instances. The authors of~\cite{BiqBin, HrgaPovh}
    consider similar instances for larger graphs with $n = 180$, leading to an additional set of 90 instances.

\item \textbf{Spinglass instances}: Toroidal 2D-grid graphs arise in physics when computing ground states for Ising spinglasses, see e.g.,~\cite{DeSimone}. Those instances are generated using the rudy graph generator~\cite{rudygenerator}. In particular,  ${\rm spinglass2pm}_{n_t,n_r}$ is a toroidal 2D-grid for a spinglass model with weights  $\{-1, 1\}$. The
grid has size $n_t \times n_r$. The percentage of edges with negative weights is $50\%$.
    
\item   \textbf{Band instances:} These instances are considered in~\cite{Fakhimi2022TheMK,Hijazi}. Given the number of vertices $n$ and the value of $k$, we let the edge set of a band graph be defined as $E:= \{ \{i,j\}  : j-i \leq k+1,\,\, 1 \leq i<j \leq n
\}$. In these graphs, $50\%$ of edge weights are set to $-1$ and the others are set to $1$. 

\end{itemize}
\subsubsection{Bounding approaches}

In Section~\ref{Sect:MaxKCut} we have shown that the max-$k$-cut problem can be formulated as the discrete SDP~\eqref{Prob:MaxKCutISDP}. The Lagrangian dual problem obtained from this formulation can be solved by any of the {alternative} approaches presented in Section~\ref{Sect:SolvingLagDual}. Note that these approaches were designed for minimization problems, so we multiply the objective function of~\eqref{Prob:MaxKCutISDP} by $-1$ in order to view it as a minimization problem. In our numerical tests, we consider the following bounding approaches for the max-$k$-cut problem: 
\begin{itemize}
    \item \textbf{SDP-B}: This approach refers to solving the basic SDP relaxation of the max-$k$-cut problem, i.e.,~\eqref{MaxKCutSDP}, using the alternating direction method of multipliers (ADMM).
    We follow a similar implementation as  described in~\cite[Section 3.1]{DeMeijerEtAl} for the graph partition problem.

    \item \textbf{LD}: By LD we refer to the Lagrangian dual bound obtained from the discrete SDP~\eqref{Prob:MaxKCutISDP} after dualizing $X \succeq \bold{0}$. Observe that~\eqref{Prob:MaxKCutISDP} fits in the  form~\eqref{Def:LDforBSDP}, where $B = \{\frac{-1}{k-1},1\}$  and $\widehat{\mathcal{D}}^{m_\ell}$ as given in~\eqref{Def:setDhat}. To obtain an edge-packing $\mathcal{K}^p$ that is likely to provide strong bounds, we use the heuristic that is described in Section~\ref{Subsec:EdgePacking}. The resulting Lagrangian dual problem is now solved using the three approaches described in Section~\ref{Sect:SolvingLagDual}:
    \begin{itemize}
    \item \textbf{LD-DSG}: This approach refers to solving the Lagrangian dual problem using the projected-deflected subgradient algorithm discussed in Section~\ref{Subsec:ProjectedSubgradient}. For the deflection parameter $\varphi_1^\ell$, see \eqref{dualUpdate2}, (observe that $\varphi^\ell_2$ does not exist as we linearize only the PSD constraint), we use the update scheme of Sherali and Ulular~\cite{SheraliUlular}. Moreover, we use Polyak's stepsize update~\cite{Polyak}, see~\eqref{stepsize}, where $U^*$ is computed using the heuristic introduced in Section~\ref{Subsec:lowerbound}. For the value of $\mu_1^\ell$ we use the approach of Held and Karp~\cite{HeldKarp}, implying that $\mu_1^0 = 1$ and we halve its value each time the obtained bound did not increase for $N_\text{step} = 40$ subsequent iterations.
    As starting point $S^0$, we use the (approximate) dual solution that we obtain from the implementation of the ADMM mentioned above.
    For the stopping criteria, we use $N_\text{stag} = 100$ and $\varepsilon_1 = 0.01$ (observe that $\varepsilon_2$ does not exist). These values are based on preliminary experiments. 

    \item \textbf{LD-ASG}: This approach refers to solving the Lagrangian dual problem using the projected-accelerated subgradient algorithm discussed in Section~\ref{Subsec:Nesterov}. The parameters in this approach are $\alpha^\ell$, $\eta_\ell$ and $\mu^\ell_1$ (again, $\beta^\ell$ does not exist, since we do only dualize the PSD constraint). The sequence $\{\eta_\ell\}_\ell^\infty$ is fully determined after taking $\eta_0 = 0$. The parameter $\alpha^\ell$ is computed via~\eqref{stepsize}, where $U^*$ is chosen as described in Section~\ref{Subsec:lowerbound}. With respect to the parameter $\mu^\ell_1$, we discriminate between two cases:
    We take $\mu^\ell_1 = 0.025$ for $n \leq 100$, whereas $\mu^\ell_1 = 0.012$ if $n > 100$. The reason to lower the value of $\mu^\ell_1$ for larger instances is to reduce the oscillating behavior at the start of the algorithm  that is heavier for  larger instances. We use the same starting point $S^0$ as has been used for LD-DSG.
    For the stopping criteria, we take $\varepsilon_1 = 1\mathrm{e-}3$ and $N_\text{stag} = 15$, which are based on preliminary tests.

    \item \textbf{LD-Bundle}: This approach refers to solving the Lagrangian dual problem using the projected bundle  algorithm described in Section~\ref{Subsec:Projectedbundle}. 
    As starting point for the proximity parameter $t$, we use $t = \kappa |z_{SDP} - U^*|/||X^*||_F$, where $\kappa > 0$ is a constant, $X^*$ is an approximate solution to the SDP relaxation~\eqref{MaxKCutSDP},  $z_{SDP}$ its corresponding objective value and $U^*$ is the bound as computed by the heuristic of Section~\ref{Subsec:lowerbound}. We take $\kappa$  equal to $0.1$ if $n \leq 100$ and $\kappa = 0.2$ otherwise. The parameter $\rho$ for deciding whether we take a serious or a null step is set to $\frac{1}{2}$. 

    The problem~\eqref{MinMaxBundle} is solved iteratively, that is, we iteratively solve with respect to either $\sigma$ or $Q$, while keeping the other variable fixed.
    We continue doing this until the norm of the difference between consecutive values of $\sigma$ is less than $1\mathrm{e}{-4}$, or after the maximum number of 20 iterations are reached. The entire bundle algorithm terminates whenever the function value of the current center point and the function value of the subgradient model at the trial point is smaller than $1\mathrm{e}{-3}$ or if the maximum number of 600 iterations is reached.

    \end{itemize}
\end{itemize}
\subsubsection{Edge-packing heuristic} \label{Subsec:EdgePacking}
    Recall that the Lagrangian dual bound depends on the edge-packing $\mathcal{K}^p$. To determine an edge-packing that is likely to provide a strong Lagrangian dual bound, we apply the following heuristic. Let $X^*$ denote an (approximate) solution to the SDP relaxation~\eqref{MaxKCutSDP}. For each triple of vertices $i,j,h \in V$, we define 
    \begin{align*}
        v^{\Delta}_{ijh} & := (X^*_{ij} + X_{jh}^* - X_{ih}^* -1)^+ + (X^*_{ij} - X_{jh}^* + X_{ih}^* -1)^+ + (-X^*_{ij} + X_{jh}^* + X_{ih}^* -1)^+         
    \intertext{and for each $S\subseteq V$ with $|S| = k+1$, we define }
        v^{C}_{S} & := \left( -\frac{k}{2} - \sum_{i,j \in S, i < j }X^*_{ij}\right)^+,
    \end{align*}
    where $(\cdot )^+ = \max( 0, \cdot)$. Indeed, $v_{ijh}^\Delta$ and $v_S^C$ provide  measures of the violation of the triangle- and clique-inequalities~\eqref{triangleIneq} and~\eqref{cliqueIneq}, respectively. When a triple $(i,j,h)$ or a subset $S$ violates these inequalities by a large amount, we have an incentive to include this triple/set within the edge-packing $\mathcal{K}^p$, as we expect its inclusion to have a positive effect on the resulting bound. Now, we proceed as follows. We let $\overline{E}$ denote the set of edges that are already included in the packing, where we initially set $\overline{E} = \emptyset$. We start with a set $K$ containing the three vertices $i,j,h$ that have the largest positive value $v^\Delta_{ijh}$. We greedily add to $K$ the vertex $t$ that maximizes the total violation 
    \begin{align*}
            \sum_{i, j \in K, i < j}v_{tij}^\Delta + \sum_{\substack{S \subseteq [n] : |S| = k+1, \\
            t \in S, S \setminus t \subseteq K}} v_S^C 
    \end{align*}
    associated with vertex $t$ and the vertices already in $K$. We restrict ourselves to the vertices $t$ such that $\{(t,k) \, : \, \, k \in K\}$ does not intersect with $\overline{E}$, preventing an edge from being packed twice. 
    We continue adding vertices to $K$ until $|K| = p$ or no more vertex with positive violation can be found. Then, we add $K$ to the edge-packing $\mathcal{K}^p$ and add to $\overline{E}$ all edges in the subgraph induced by $K$. We repeat this process starting from three vertices whose induced subgraph is not yet covered by $\overline{E}$, until $\mathcal{K}^p$ contains a maximum number of subsets. This maximum is in our computations set equal to $5n$.

\subsubsection{Computation of near-optimal {\em k}-partitions} \label{Subsec:lowerbound}
    To apply Polyak's stepsize update~\cite{Polyak}, see~\eqref{stepsize}, one needs a lower bound $U^*$ on the optimal value of the max-$k$-cut problem. Hence, we aim to find near-optimal feasible solutions to the max-$k$-cut problem.  We proceed as follows. Let $X^*$ denote an (approximate) solution to the SDP relaxation~\eqref{MaxKCutSDP} of the max-$k$-cut problem. The approximation algorithm introduced by Frieze and Jerrum~\cite{Frieze1995ImprovedAA}, in the sequel denoted by the FJ algorithm, applies a rounding strategy on $X^*$ to obtain a feasible max-$k$-cut. As the algorithm is probabilistic, we perform 500 independent trials of the FJ algorithm and take its best solution. 
    
    Next, we try to improve this feasible solution using a variable depth search~\cite{KernighanLin}. Let $V_0 = V$ and compute for each vertex $v \in V_0$ the solution we obtain from moving $v$ to its best alternative cluster (i.e., the one that leads to the largest cut value). Let $v_1$ denote the vertex which movement leads to the best improvement (even if the cost improvement is negative) and delete $v_1$ from $V_0$. We repeat the procedure until $V_0$ is empty, which leads to an ordered list of vertices $\{v_1, \ldots, v_n\}$. Finally, we compute the value of $T \in [n]$ such that the simultaneous movement of the vertices $\{v_1, \ldots, v_T\}$ to their best alternative cluster results in the best possible cut. The resulting feasible $k$-partition can be encoded by a matrix $Y \in \{\frac{-1}{k-1}, 1\}^{n \times n}$ where   $Y_{ij} = \frac{-1}{k-1}$  if $i$ and $j$ are in different subsets of the partition, and $Y_{ij}=1$ otherwise.
    Inspired by the heuristic of Rendl et al.~\cite{RendlEtAl}, we  replace $X^*$ by $\omega X^* + (1-\omega) Y$, where $\omega \in (0,1)$, and repeat the entire process again using this new solution $X^*$. We continue doing so until no better solution can be found by either the FJ algorithm or the variable-depth search. Based on preliminary tests, we take $\omega = 0.8$.

\subsection{Computational results for the max-cut problem}
\label{Subsec:numeric2cut}
In this section we present computational results for the max-cut problem.

Table~\ref{Tab:rudy_max2cut_values} shows average bound values for the max-cut problem on several types of Rudy instances. The column `$n$' denotes the number of vertices, whereas the columns `SDP-B' and `OPT' denote the SDP bound~\eqref{MaxKCutSDP} and the known optima (as presented in~\cite{HrgaPovh,BiqMac}), respectively, averaged over all 10 instances of that class. For each instance, we performed the approaches
LD-DSG, LD-ASG and LD-Bundle to compute the Lagrangian dual bound. Although the resulting bounds obtained by these approaches slightly differ, the difference is always within 1\% (relative to the minimum of the three). The columns `LD' provide this minimal value. The values in the columns `Rel.~gap closed' are computed by the formula $100\cdot (\text{SDP} - \text{LD})/(\text{SDP}-\text{OPT})$ and denote the percentage of the gap between the SDP bound and the optimum that is closed by the Lagrangian dual bound. In Table~\ref{Tab:rudy_max2cut_values}  we only present results for $p=3$ and $p=7$, and results for   $p$ ranging between $p = 3$ and $p = 23$ are given in Figure~\ref{Fig:max2cut_ImprovementP}.
The corresponding computation times (in wall clock time) are presented in Table~\ref{Tab:rudy_max2cut_times}. Here we do make a distinction between the approaches LD-ASG, LD-ASG and LD-Bundle.

Observe that Table~\ref{Tab:rudy_max2cut_values} and~\ref{Tab:rudy_max2cut_times} only include a subset of the instance types of the Rudy instances. The reason for this is that for the large instances with $n = 180$, optima are known for only six out of nine instance types~\cite{HrgaPovh}. To keep the presentation balanced, we only present results for six instance types of the small instances as well.

\begin{table}[h]
\footnotesize
\setlength{\tabcolsep}{5pt}
\centering
\begin{tabular}{@{}rrSSSSSS@{}}
\toprule
\multicolumn{1}{c}{\textbf{}}                                                         & \multicolumn{1}{c}{\textbf{}}    & \multicolumn{1}{c}{\textbf{}}    & \multicolumn{1}{c}{\textbf{}}    & \multicolumn{2}{c}{\textbf{LD bound for $\mathbf{p = 7}$}}                                                                              & \multicolumn{2}{c}{\textbf{LD bound for $\mathbf{p = 17}$}}                                                                             \\ \cmidrule(l){5-6} \cmidrule(l){7-8} 
\multicolumn{1}{c}{\textbf{\begin{tabular}[c]{@{}c@{}}Instance\\ Class\end{tabular}}} & \multicolumn{1}{c}{\textbf{$\mathbf{n}$}} & \multicolumn{1}{c}{\textbf{SDP-B}} & \multicolumn{1}{c}{\textbf{OPT}} & \multicolumn{1}{c}{\textbf{LD}} & \multicolumn{1}{c}{\textbf{\begin{tabular}[c]{@{}c@{}}Rel. gap \\ closed (\%)\end{tabular}}} & \multicolumn{1}{c}{\textbf{LD}} & \multicolumn{1}{c}{\textbf{\begin{tabular}[c]{@{}c@{}}Rel. gap \\ closed (\%)\end{tabular}}} \\ \midrule
g05\_80                                                                               & 80                               & 950.4                            & 929.1                            & 940.2                           & 47.7                                                                                         & 939.5                           & 51.4                                                                                         \\
pm1d\_80                                                                              & 80                               & 297.3                            & 254.5                            & 276.9                           & 47.6                                                                                         & 274.9                           & 52.3                                                                                         \\
pm1s\_100                                                                             & 100                              & 139.7                            & 122.6                            & 131.7                           & 46.5                                                                                         & 131.1                           & 50.1                                                                                         \\
pw05\_100                                                                             & 100                              & 8364.7                           & 8147.1                           & 8273.8                          & 41.8                                                                                         & 8268.1                          & 44.4                                                                                         \\
w01\_100                                                                              & 100                              & 790.7                            & 699.9                            & 742.3                           & 53.3                                                                                         & 739.1                           & 56.9                                                                                         \\
w09\_100                                                                              & 100                              & 2544.8                           & 2176.9                           & 2381.8                          & 44.3                                                                                         & 2372.3                          & 46.9                                                                                         \\ \midrule
g05\_180                                                                              & 180                              & 4582.6                           & 4494.6                           & 4557.7                          & 28.2                                                                                         & 4557.8                          & 28.2                                                                                         \\
pm1d\_180                                                                             & 180                              & 1042.3                           & 870.5                            & 990.2                           & 30.3                                                                                         & 989.0                           & 31.0                                                                                         \\
pw05\_180                                                                             & 180                              & 25952.2                          & 25347.3                          & 25771.8                         & 29.8                                                                                         & 25765.7                         & 30.8                                                                                         \\
pw09\_180                                                                             & 180                              & 43515.2                          & 42964.7                          & 43355.1                         & 29.1                                                                                         & 43350.5                         & 29.9                                                                                         \\
w05\_180                                                                              & 180                              & 4671.1                           & 3913.1                           & 4447.6                          & 29.5                                                                                         & 4441.5                          & 30.3                                                                                         \\
w09\_180                                                                              & 180                              & 6173.5                           & 5167.8                           & 5888.6                          & 28.3                                                                                         & 5879.4                          & 29.3                                                                                         \\ \bottomrule
\end{tabular}
\caption{Average bound values (SDP-B, LD and optimum) for the max-cut problem on Rudy instances. Each reported value is the average over 10 randomly generated instances of that type. \label{Tab:rudy_max2cut_values}}
\end{table}

\begin{table}[h]
\centering
\footnotesize
\begin{tabular}{@{}rrSSSSSSS@{}}
\toprule
\multicolumn{1}{c}{\textbf{}}                                                         & \multicolumn{1}{c}{\textbf{}}             & \multicolumn{1}{c}{\textbf{}}    & \multicolumn{3}{c}{\textbf{LD bound for $\mathbf{p = 7}$}}                                                         & \multicolumn{3}{c}{\textbf{LD bound for $\mathbf{p = 17}$}}                                                        \\ \cmidrule(l){4-6} \cmidrule(l){7-9}
\multicolumn{1}{c}{\textbf{\begin{tabular}[c]{@{}c@{}}Instance\\ Class\end{tabular}}} & \multicolumn{1}{c}{\textbf{$\mathbf{n}$}} & \multicolumn{1}{c}{\textbf{SDP-B}} & \multicolumn{1}{c}{\textbf{LD-DSG}} & \multicolumn{1}{c}{\textbf{LD-ASG}} & \multicolumn{1}{c}{\textbf{LD-Bundle}} & \multicolumn{1}{c}{\textbf{LD-DSG}} & \multicolumn{1}{c}{\textbf{LD-ASG}} & \multicolumn{1}{c}{\textbf{LD-Bundle}} \\ \midrule
g05\_80                                                                               & 80                                        & 0.19                             & 1.05                                & 0.99                                & 1.56                                   & 12.07                               & 11.42                               & 5.87                                   \\
pm1d\_80                                                                              & 80                                        & 0.10                             & 1.10                                & 0.87                                & 1.49                                   & 13.33                               & 10.81                               & 5.98                                   \\
pm1s\_100                                                                             & 100                                       & 0.92                             & 2.52                                & 2.64                                & 3.21                                   & 17.68                               & 18.45                               & 8.57                                   \\
pw05\_100                                                                             & 100                                       & 0.28                             & 2.82                                & 1.61                                & 5.96                                   & 26.89                               & 15.27                               & 16.15                                  \\
w01\_100                                                                              & 100                                       & 0.36                             & 3.21                                & 2.55                                & 5.61                                   & 26.81                               & 22.42                               & 14.28                                  \\
w09\_100                                                                              & 100                                       & 0.46                             & 3.23                                & 1.94                                & 6.31                                   & 30.46                               & 16.13                               & 16.68                                  \\ \midrule
g05\_180                                                                              & 180                                       & 2.74                             & 10.65                               & 10.16                               & 14.16                                  & 68.83                               & 55.45                               & 28.51                                  \\
pm1d\_180                                                                             & 180                                       & 1.54                             & 11.10                               & 9.42                                & 12.88                                  & 84.46                               & 63.38                               & 30.54                                  \\
pw05\_180                                                                             & 180                                       & 0.96                             & 12.62                               & 9.03                                & 27.29                                  & 101.44                              & 58.83                               & 50.01                                  \\
pw09\_180                                                                             & 180                                       & 0.96                             & 11.65                               & 8.09                                & 27.95                                  & 96.89                               & 55.70                               & 50.84                                  \\
w05\_180                                                                              & 180                                       & 0.92                             & 12.50                               & 9.16                                & 23.90                                  & 99.04                               & 59.71                               & 49.86                                  \\
w09\_180                                                                              & 180                                       & 0.91                             & 12.76                               & 8.96                                & 25.63                                  & 105.31                              & 54.41                               & 50.78                                  \\ \bottomrule
\end{tabular}
\caption{Average computation times (wall clock time in seconds) of Lagrangian dual approaches for the max-cut problem on Rudy instances. Each reported value is the average over 10 randomly generated instances of that type. \label{Tab:rudy_max2cut_times}}
\end{table}

It can be observed from Table~\ref{Tab:rudy_max2cut_values} that the Lagrangian dual bound provides a significant improvement over the bound~\eqref{MaxKCutSDP}. The relative gap closed by the Lagrangian dual for the instances with $n \leq 100$ is in the range 41\%-57\%. For the larger instances with $n = 180$, this gap is in the range 28\%-31\%. This drop aligns with our expectations, since if $n$ increases, the number of subsets of $[n]$ with size at most $p$ also increases. Hence, the relative number of subsets $K_\ell \subseteq [n]$ for which the constraint $X[K_\ell] \in \widehat{\mathcal{D}}^{m_\ell}$ is included in the discrete SDP becomes lower as $n$ increases, leading to a weaker bound. When $p$ increases, we observe that the Lagrangian dual bound becomes stronger, which is indeed in line with Theorem~\ref{THM:HierarchyBounds}. This comes at the cost of computation time, since the evaluation of the dual function $g(S,\lambda)$ becomes more costly. Indeed, this evaluation involves a full enumeration over the elements in $\widehat{\mathcal{D}}^{m_\ell}$, whose cardinality increases (drastically) with $p$. This effect is clearly visible in Table~\ref{Tab:rudy_max2cut_times}.  

When comparing the various approaches for computing the Lagrangian dual bound, we observe from Table~\ref{Tab:rudy_max2cut_times} that  for small $p$ LD-ASG is the fastest approach on almost all instance types, although LD-DSG follows on a short distance. Apparently, the accelerated two-step approach by Nesterov~\cite{Nesterov1983AMF} pays off, which is reflected by a lower number of iterations required to converge compared to LD-DSG.  For example, for Rudy instances with 
$n=100$, LB-DSG requires  1264 iterations on average to converge, while LD-ASG requires only 884.
The projected-bundle method LD-Bundle yields the largest computation times for $p = 7$. We observe in our experiments that the number of iterations needed by the bundle method is much lower than for the subgradient methods, but one iteration is in general more costly, leading to a larger total computation time. In particular, the bundle method requires  401 iterations on average to converge  for Rudy instances with $n=100$.
At the same time, it seems that the bundle method is more robust against (small) changes in the parameters, whereas such changes cause a large effect on the performance of the subgradient methods. This is a well-known short-coming of first-order methods, see e.g.,~\cite{Frangioni2017OnTC}. This robustness is also reflected by our experiments for larger values of $p$. As explained before, $p$ has a negative effect on the computation time of the Lagrangian dual bound. This effect is mainly due to an increase in the computation time per iteration due to a full enumeration. Since the bundle method needs the least number of iterations, the approach LD-Bundle becomes the favored method when $p$ becomes large. 

To further study the effect of $p$ on the Lagrangian dual bound, we present in Figure~\ref{Fig:max2cut_ImprovementP} the (average) relative gap closed and the (average) computation time per instance type for different values of $p$. For $n \leq 100$, we test for $3 \leq p \leq 23$. For $p$ larger than 23, we observe that we can no longer perform the computations within one hour. For the large Rudy instances with $n = 180$, we observe that this frontier is reached earlier, hence we present results for $3 \leq p \leq 19$. Different from Table~\ref{Tab:rudy_max2cut_values} and~\ref{Tab:rudy_max2cut_times}, we now take into account all instance types for which optimal values are known. The Lagrangian dual bounds are computed via the approach LD-DSG, but results obtained by the other two algorithms are similar. 

Figure~\ref{Fig:max2cut_ImprovementP} demonstrates that the value of $p$ has a positive impact on the quality of the bound, resulting in a significant reduction in the gap between the obtained bounds and the optimal solution. We observe that the marginal improvement over $p$ is, however, diminishing. That is, an improvement in $p$ from 3 to 5 has a large effect on the quality of the Lagrangian dual bounds, whereas this positive effect tails off when $p$ becomes larger. At the same time, the computation times remain relatively small for values of $p$ up to 20 (resp.~15) for the small (resp.~large) instance. For larger values of $p$, we observe that the computation times rapidly grow due to the extremely large cardinality of the involved sets $\widehat{D}^{m_\ell}$, see Appendix~\ref{Appendix:SetPartitions}. 

Our bounds can be further improved by considering overlapping submatrices
in \eqref{Def:LDforBSDPequivalent},  at the cost of increased computational effort. 
However, one can also improve the basic SDP relaxation \eqref{MaxKCutSDP} by adding various valid inequalities.
 The authors of \cite{RendlEtAl} strengthen the basic SDP relaxation for the max-cut problem by adding the triangle inequalities and solve the resulting relaxation using a dynamic version of the bundle method. Their bounds dominate ours.

\begin{figure}[h]
    \centering

    \begin{subfigure}[b]{0.49\textwidth}
    \centering
        \includegraphics[scale=0.5]{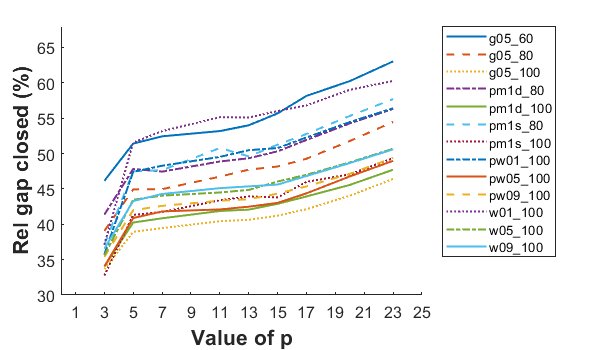}
        \caption{Rel.~gap closed for the max-cut problem on small Rudy instances}
    \end{subfigure} \hfill 
    \begin{subfigure}[b]{0.49\textwidth}
    \centering
        \includegraphics[scale=0.5]{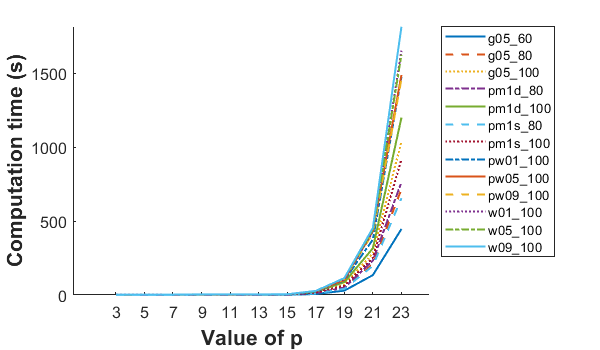}
        \caption{Computation times for the max-cut problem on small Rudy instances}
    \end{subfigure} 
    \begin{subfigure}[b]{0.49\textwidth}
    \centering
        \includegraphics[scale=0.5]{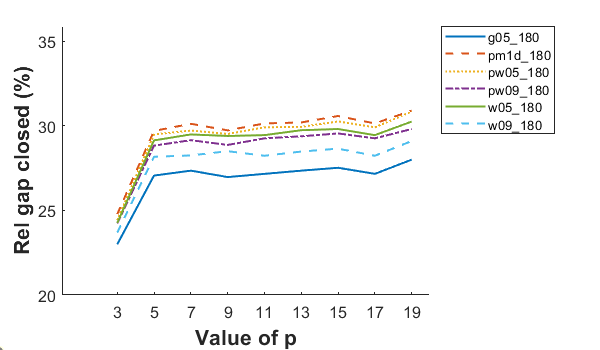}
        \caption{Rel.~gap closed for the max-cut problem on large Rudy instances}
    \end{subfigure} \hfill 
    \begin{subfigure}[b]{0.49\textwidth}
    \centering
        \includegraphics[scale=0.5]{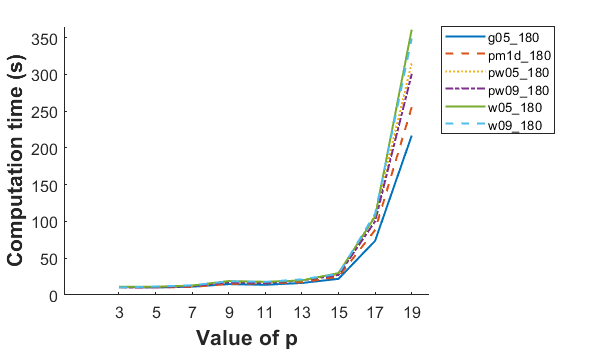}
        \caption{Computation times for the max-cut problem on large Rudy instances}
    \end{subfigure}
    \caption{Average relative gap (between SDP~\eqref{MaxKCutSDP} and optimum) closed and average computation times of the approach LD-DSG for different values of $p$. Results are presented for the max-cut problem on small and large Rudy instances.   \label{Fig:max2cut_ImprovementP}}
\end{figure}

\subsection{Computational results for the max-3-cut problem} \label{Subsec:numeric3cut}
Although there are several studies on solving the max-$k$-cut problem for  $k>2$ \cite{Anjos2013,SousaAnjos,Sousa2020ComputationalSO,Lu2021ABA},  they do not provide bounds at the root node.
Here, we discuss our findings for the max-$3$-cut problem. For this problem we first perform experiments on the same instances as for the max-cut problem, leading to Table~\ref{Tab:rudy_max3cut_values} and~\ref{Tab:rudy_max3cut_times} and Figure~\ref{Fig:max3cut_ImprovementP}. Although the set-up of the experiments and the explanation of the columns in the tables is the same, there are two differences. First, due to the larger computation times of the Lagrangian dual bounds for the max-3-cut problem, we test up to $p \leq 15$ (resp.~$p \leq 13$) for the smaller (resp.~larger) instances. For that reason, we report results for $p = 13$ instead of $p = 17$ in Table~\ref{Tab:rudy_max3cut_values} and~\ref{Tab:rudy_max3cut_times}. These larger computation times can be explained from the fact that the sets $\widehat{\mathcal{D}}^{m_\ell}$ have a (significantly) larger cardinality when $k$ increases, see Appendix~\ref{Appendix:SetPartitions}. Moreover, since optimal solutions for the max-3-cut problem and those instances are not available in the literature, we instead report the lower bounds as obtained by the procedure explained in Section~\ref{Subsec:lowerbound}. Consequently, the relative gap closed is computed by the formula $100\cdot (\text{SDP} - \text{LD})/(\text{SDP}-\text{LB})$.

When considering Table~\ref{Tab:rudy_max3cut_values} and~\ref{Tab:rudy_max3cut_times}, we draw very similar conclusions as for the max-cut problem. Again, the Lagrangian dual bounds provide a substantial improvement compared to the basic SDP bound~\eqref{MaxKCutSDP}, although the relative gaps closed are smaller than for the max-cut problem. This can be explained by the fact that we compare to a lower bound instead of the optimum.

The projected-accelerated version of the subgradient algorithm seems to perform best for small values of $p$, whereas the projected-bundle method becomes favorable for large values of $p$. The conclusions drawn from Figure~\ref{Fig:max3cut_ImprovementP} are similar to those drawn from Figure~\ref{Fig:max2cut_ImprovementP}. 

We also conduct experiments on  the Band instances~\cite{Fakhimi2022TheMK,Hijazi} in Table~\ref{Tab:Band_max3cut}. For those instances, the optimal values are given in \cite{Fakhimi2022TheMK}. Note that for the instances band50\_3, band100\_3 and band150\_3, our bounds for $p=13$, when rounded down, provide optimal solutions to the problems.

\begin{table}[h]
\centering\footnotesize
\begin{tabular}{@{}rrSSSSSS@{}}
\toprule
\multicolumn{1}{c}{\textbf{}}                                                         & \multicolumn{1}{c}{\textbf{}}             & \multicolumn{1}{c}{\textbf{}}    & \multicolumn{1}{c}{\textbf{}}   & \multicolumn{2}{c}{\textbf{LD bound for $\mathbf{p = 7}$}}                                                                    & \multicolumn{2}{c}{\textbf{LD bound for $\mathbf{p = 13}$}}                                                                   \\ \cmidrule(l){5-6} \cmidrule(l){7-8}  
\multicolumn{1}{c}{\textbf{\begin{tabular}[c]{@{}c@{}}Instance\\ Class\end{tabular}}} & \multicolumn{1}{c}{\textbf{$\mathbf{n}$}} & \multicolumn{1}{c}{\textbf{SDP-B}} & \multicolumn{1}{c}{\textbf{LB}} & \multicolumn{1}{c}{\textbf{LD}} & \multicolumn{1}{c}{\textbf{\begin{tabular}[c]{@{}c@{}}Rel. gap\\ closed (\%)\end{tabular}}} & \multicolumn{1}{c}{\textbf{LD}} & \multicolumn{1}{c}{\textbf{\begin{tabular}[c]{@{}c@{}}Rel. gap\\ closed (\%)\end{tabular}}} \\ \midrule
g05\_80                                                                               & 80                                        & 1251.5                           & 1215.0                          & 1246.4                          & 13.8                                                                                        & 1245.8                          & 15.4                                                                                        \\
pm1d\_80                                                                              & 80                                        & 368.8                            & 296.8                           & 356.5                           & 17.0                                                                                        & 355.7                           & 18.2                                                                                        \\
pm1s\_100                                                                             & 100                                       & 172.6                            & 142.2                           & 167.5                           & 16.6                                                                                        & 167.0                           & 18.2                                                                                        \\
pw05\_100                                                                             & 100                                       & 11007.3                          & 10620.2                         & 10962.5                         & 11.6                                                                                        & 10957.9                         & 12.8                                                                                        \\
w01\_100                                                                              & 100                                       & 954.8                            & 807.9                           & 922.2                           & 22.2                                                                                        & 918.9                           & 24.4                                                                                        \\
w09\_100                                                                              & 100                                       & 3171.1                           & 2568.4                          & 3076.4                          & 15.7                                                                                        & 3068.0                          & 17.1                                                                                        \\ \midrule
g05\_180                                                                              & 180                                       & 6074.1                           & 5906.5                          & 6061.2                          & 7.7                                                                                         & 6060.6                          & 8.1                                                                                         \\
pm1d\_180                                                                             & 180                                       & 1316.2                           & 997.6                           & 1287.5                          & 9.0                                                                                         & 1285.9                          & 9.5                                                                                         \\
pw05\_180                                                                             & 180                                       & 34320.1                          & 33191.3                         & 34230.3                         & 8.0                                                                                         & 34223.9                         & 8.5                                                                                         \\
pw09\_180                                                                             & 180                                       & 57789.1                          & 56744.6                         & 57699.9                         & 8.5                                                                                         & 57694.8                         & 9.0                                                                                         \\
w05\_180                                                                              & 180                                       & 5923.4                           & 4617.3                          & 5789.4                          & 10.3                                                                                        & 5783.3                          & 10.7                                                                                        \\
w09\_180                                                                              & 180                                       & 7838.3                           & 5988.2                          & 7671.3                          & 9.0                                                                                         & 7660.4                          & 9.6                                                                                         \\ \bottomrule
\end{tabular}
\caption{Average bound values (SDP-B, LD and LB) for the max-3-cut problem on Rudy instances. Each reported value is the average over 10 randomly generated instances of that type. \label{Tab:rudy_max3cut_values}}
\end{table}

\begin{table}[h]
\centering
\footnotesize
\begin{tabular}{@{}rrSSSSSSS@{}}
\toprule
\multicolumn{1}{c}{\textbf{}}                                                         & \multicolumn{1}{c}{\textbf{}}             & \multicolumn{1}{c}{\textbf{}}    & \multicolumn{3}{c}{\textbf{LD bound for $\mathbf{p = 7}$}}                                                         & \multicolumn{3}{c}{\textbf{LD bound for $\mathbf{p = 13}$}}                                                        \\ \cmidrule(l){4-6} \cmidrule(l){7-9}
\multicolumn{1}{c}{\textbf{\begin{tabular}[c]{@{}c@{}}Instance\\ Class\end{tabular}}} & \multicolumn{1}{c}{\textbf{$\mathbf{n}$}} & \multicolumn{1}{c}{\textbf{SDP-B}} & \multicolumn{1}{c}{\textbf{LD-DSG}} & \multicolumn{1}{c}{\textbf{LD-ASG}} & \multicolumn{1}{c}{\textbf{LD-Bundle}} & \multicolumn{1}{c}{\textbf{LD-DSG}} & \multicolumn{1}{c}{\textbf{LD-ASG}} & \multicolumn{1}{c}{\textbf{LD-Bundle}} \\ \midrule
g05\_80                                                                               & 80                                        & 0.07                             & 1.50                                & 0.85                                & 1.44                                   & 41.73                               & 22.46                               & 17.26                                  \\
pm1d\_80                                                                              & 80                                        & 0.05                             & 1.81                                & 0.80                                & 1.81                                   & 48.40                               & 22.85                               & 17.55                                  \\
pm1s\_100                                                                             & 100                                       & 0.34                             & 2.74                                & 2.26                                & 3.17                                   & 55.62                               & 43.95                               & 27.23                                  \\
pw05\_100                                                                             & 100                                       & 0.23                             & 4.10                                & 1.56                                & 5.09                                   & 96.14                               & 33.50                               & 38.28                                  \\
w01\_100                                                                              & 100                                       & 0.20                             & 4.17                                & 2.36                                & 5.01                                   & 85.00                               & 51.36                               & 32.39                                  \\
w09\_100                                                                              & 100                                       & 0.41                             & 4.86                                & 1.95                                & 5.19                                   & 110.85                              & 34.80                               & 39.82                                  \\ \midrule
g05\_180                                                                              & 180                                       & 0.81                             & 11.51                               & 7.35                                & 13.66                                  & 352.37                              & 209.35                              & 147.64                                 \\
pm1d\_180                                                                             & 180                                       & 0.57                             & 13.82                               & 7.79                                & 14.63                                  & 414.38                              & 210.28                              & 159.96                                 \\
pw05\_180                                                                             & 180                                       & 0.94                             & 14.58                               & 7.78                                & 17.21                                  & 500.36                              & 204.30                              & 162.29                                 \\
pw09\_180                                                                             & 180                                       & 0.84                             & 15.75                               & 7.36                                & 17.58                                  & 526.15                              & 197.02                              & 168.00                                 \\
w05\_180                                                                              & 180                                       & 1.18                             & 16.64                               & 8.44                                & 17.76                                  & 504.47                              & 220.80                              & 160.83                                 \\
w09\_180                                                                              & 180                                       & 1.48                             & 17.98                               & 8.77                                & 19.77                                  & 567.58                              & 215.09                              & 180.73                                 \\ \bottomrule
\end{tabular}
\caption{Average computation times of Lagrangian dual approaches for the max-3-cut problem on Rudy instances. Each reported value is the average over 10 randomly generated instances of that type. \label{Tab:rudy_max3cut_times}}
\end{table}

\begin{figure}[h]
    \centering
    \begin{subfigure}[b]{0.49\textwidth}
    \centering
        \includegraphics[scale=0.5]{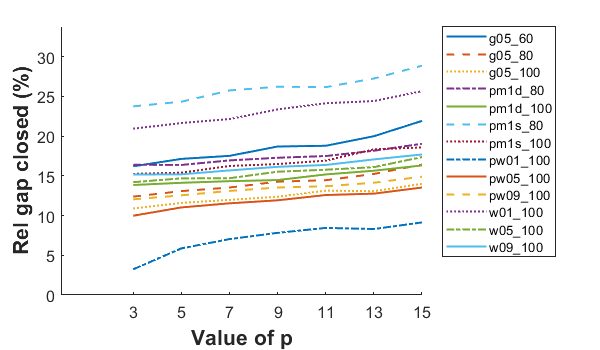}
        \caption{Rel.~gap closed for the max-3-cut problem on small Rudy instances}
    \end{subfigure} \hfill 
    \begin{subfigure}[b]{0.49\textwidth}
    \centering
        \includegraphics[scale=0.5]{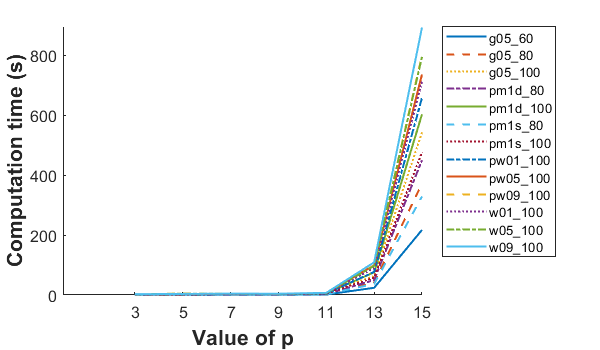}
        \caption{Computation times for the max-3-cut problem on small Rudy instances}
    \end{subfigure} 
    \begin{subfigure}[b]{0.49\textwidth}
    \centering
        \includegraphics[scale=0.5]{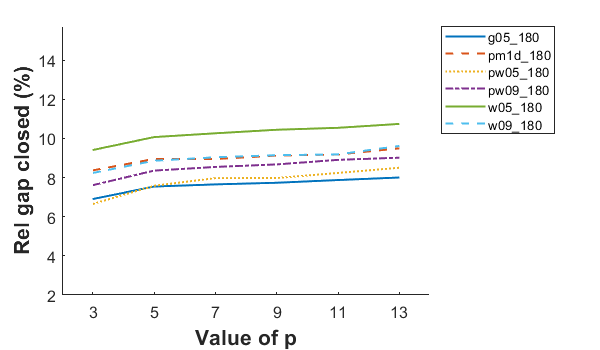}
        \caption{Rel.~gap closed for the max-3-cut problem on large Rudy instances}
    \end{subfigure} \hfill 
    \begin{subfigure}[b]{0.49\textwidth}
    \centering
        \includegraphics[scale=0.5]{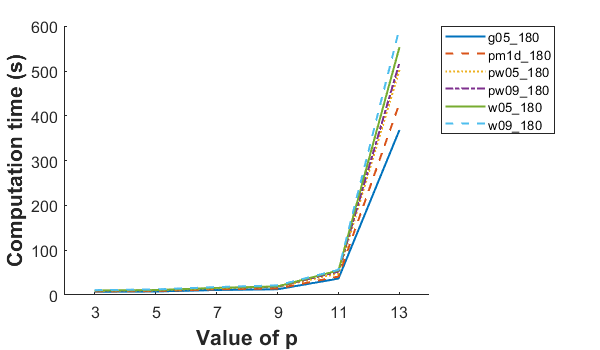}
        \caption{Computation times for the max-3-cut problem on large Rudy instances}
    \end{subfigure}
    \caption{Relative gap closed and computation times of the approach LD-DSG for different values of $p$. Results are presented for the max-$3$-cut problem on small and large Rudy instances.    \label{Fig:max3cut_ImprovementP}}
\end{figure}

\begin{table}[h]
\footnotesize
\centering
\begin{tabular}{@{}rrrrrrrrrrrr@{}}
\toprule
\multicolumn{1}{c}{\textbf{}}                                                          & \multicolumn{1}{c}{\textbf{}}             & \multicolumn{1}{c}{\textbf{}}             & \multicolumn{1}{c}{\textbf{}}             & \multicolumn{1}{c}{\textbf{}}    & \multicolumn{1}{c}{\textbf{}}    & \multicolumn{3}{c}{\textbf{LD bound for $\mathbf{p = 3}$}}                                                                                                                                                            & \multicolumn{3}{c}{\textbf{LD bound for $\mathbf{p = 13}$}}                                                                                                                                                            \\ \cmidrule(l){7-9} \cmidrule(l){10-12}
\multicolumn{1}{c}{\textbf{\begin{tabular}[c]{@{}c@{}}Instance \\ Class \\ \end{tabular}}} & \multicolumn{1}{c}{\textbf{$\mathbf{n}$}} & \multicolumn{1}{c}{\textbf{$\mathbf{m}$}} & \multicolumn{1}{c}{\textbf{$\mathbf{k}$}} & \multicolumn{1}{c}{\textbf{SDP-B}} & \multicolumn{1}{c}{\textbf{OPT}} & \multicolumn{1}{c}{\textbf{LD}} & \multicolumn{1}{c}{\textbf{\begin{tabular}[c]{@{}c@{}}Rel. gap\\ closed \\ (\%)\end{tabular}}} & \multicolumn{1}{c}{\textbf{\begin{tabular}[c]{@{}c@{}}Comp.\\ time (s)\end{tabular}}} & \multicolumn{1}{c}{\textbf{LD}} & \multicolumn{1}{c}{\textbf{\begin{tabular}[c]{@{}c@{}}Rel. gap\\ closed \\ (\%)\end{tabular}}} & \multicolumn{1}{c}{\textbf{\begin{tabular}[c]{@{}c@{}}Comp.\\ time (s)\end{tabular}}} \\ \midrule
band50\_3   & 50   & 190 & 3  &  53.5  &  49 &  50.8 & 60.0 & 0.46 & {\bf 49.8} & 82.8 & 25.10 \\
band100\_3  & 100  & 390 & 3  &  107.6 &  99  & 101.8 & 67.4  & 1.16 & {\bf 99.7} & 91.9 & 48.52 \\
band150\_3  &  150 & 590 & 3  &  161.7 & 150 & 155.7 & 51.3 & 4.57 & {\bf 150.6} &  94.9  & 160.88 \\
band200\_3  & 200 & 790 & 3  & 215.8 & 199 & 206.1 & 57.7 & 8.72 & 200.8 & 89.3 & 324.15 \\
band250\_3  & 250 & 990 & 3 & 269.9 & 249 & 263.5 & 30.6 & 10.01 & 255.0 & 71.3 & 312.14 \\
\bottomrule
\end{tabular}
\caption{Bounds (SDP-B, LD and optimum), relative gap closed by the Lagrangian dual bound and computation times for LD-ASG for the max-$3$-cut problem on Band instances. \label{Tab:Band_max3cut} }
\end{table}

\subsection{Computational results for the max-4-cut problem} \label{Subsec:numeric4cut}

For the max-4-cut problem, we test our approaches on the Band instances~\cite{Fakhimi2022TheMK,Hijazi}. Table~\ref{Tab:band_max4cut_values} shows the bound values, the optima as reported in~\cite{Fakhimi2022TheMK} and the relative gap closed by the Lagrangian dual bound for various values of $p$. The column names are similar as in previous sections. Table~\ref{Tab:band_max4cut_times} and Figure~\ref{Fig:band_max4cut} show the corresponding computation times of the Lagrangian dual approach, where we restrict ourselves to the approach LD-ASG, as this turned out to be the most efficient procedure for small values of $p$, see Sections~\ref{Subsec:numeric2cut} and~\ref{Subsec:numeric3cut}. 

We observe that the Lagrangian dual bounds are significantly improving over the basic SDP bound SDP-B, with a relative gap closed by the Lagrangian dual bound 
ranging between 19--38\% for small values of $p$ ($p = 3,5$) to 65--74\% for  $p = 13$. Interestingly, these gaps seem not to depend much on the value of $n$ and are high also for larger instances. Also, the diminishing marginal effect over $p$ that we observed for the rudy instances, seems not to be present for these type of instances. We are unable to compute the bounds for $n=15$  due to memory limitations. Indeed, Figure~\ref{Fig:band_max4cut} suggests that the relative gap closed is improving at a fairly constant rate over $p$. 

\begin{table}[h]
\centering \footnotesize
\begin{tabular}{@{}rrSSSSSSSS@{}}
\toprule
\multicolumn{1}{c}{\textbf{}}                                                         & \multicolumn{1}{c}{\textbf{}}             & \multicolumn{1}{c}{\textbf{}}    & \multicolumn{1}{c}{\textbf{}}    & \multicolumn{6}{c}{\textbf{Relative gap closed by LD bound (\%)}}                                                                                                                                                                                                                                           \\ \cmidrule(l){5-10} 
\multicolumn{1}{c}{\textbf{\begin{tabular}[c]{@{}c@{}}Instance Class\end{tabular}}} & \multicolumn{1}{c}{\textbf{$\mathbf{n}$}} & \multicolumn{1}{c}{\textbf{SDP-B}} & \multicolumn{1}{c}{\textbf{OPT}} & \multicolumn{1}{c}{\textbf{$\mathbf{p = 3}$}} & \multicolumn{1}{c}{\textbf{$\mathbf{p = 5}$}} & \multicolumn{1}{c}{\textbf{$\mathbf{p = 7}$}} & \multicolumn{1}{c}{\textbf{$\mathbf{p = 9}$}} & \multicolumn{1}{c}{\textbf{$\mathbf{p = 11}$}} & \multicolumn{1}{c}{\textbf{$\mathbf{p = 13}$}} \\ \midrule
band50\_4                                                                             & 50                                        & 68.9                             & 59                               & 37.5                                          & 34.2                                          & 41.0                                          & 46.9                                          & 62.1                                           & 73.6                                           \\
band100\_4                                                                            & 100                                       & 138.7                            & 117                              & 34.2                                          & 28.8                                          & 35.8                                          & 49.9                                          & 51.5                                           & 65.0                                           \\
band150\_4                                                                            & 150                                       & 208.6                            & 175                              & 20.2                                          & 29.0                                          & 35.9                                          & 48.7                                          & 53.9                                           & 64.4                                           \\
band200\_4                                                                            & 200                                       & 278.4                            & 234                              & 20.3                                          & 28.4                                          & 33.9                                          & 48.9                                          & 53.0                                           & 67.1                                           \\
band250\_4                                                                            & 250                                       & 348.3                            & 292                              & 19.9                                          & 30.3                                          & 35.3                                          & 46.5                                          & 54.6                                           & 66.4                                           \\ \bottomrule
\end{tabular}
\caption{Bound values (SDP-B and optimum) and relative gap closed by the Lagrangian dual bound for various values of $p$ for the max-4-cut problem on Band instances. \label{Tab:band_max4cut_values} }
\end{table}

\begin{table}[h]
\centering 
\footnotesize 
\begin{tabular}{@{}rrSSSSSS@{}}
\toprule
\multicolumn{1}{c}{\textbf{}}                                                         & \multicolumn{1}{c}{\textbf{}}             & \multicolumn{6}{c}{\textbf{Computation times (s) for LD-ASG}}                                                                                                                                                                                                                                   \\ \cmidrule(l){3-8} 
\multicolumn{1}{c}{\textbf{\begin{tabular}[c]{@{}c@{}}Instance Class\end{tabular}}} & \multicolumn{1}{c}{\textbf{$\mathbf{n}$}} & \multicolumn{1}{c}{\textbf{$\mathbf{p = 3}$}} & \multicolumn{1}{c}{\textbf{$\mathbf{p = 5}$}} & \multicolumn{1}{c}{\textbf{$\mathbf{p = 7}$}} & \multicolumn{1}{c}{\textbf{$\mathbf{p = 9}$}} & \multicolumn{1}{c}{\textbf{$\mathbf{p = 11}$}} & \multicolumn{1}{c}{\textbf{$\mathbf{p = 13}$}} \\ \midrule
band50\_4                                                                             & 50                                        & 0.27                                          & 0.21                                          & 0.28                                          & 0.43                                          & 8.52                                           & 172.10                                         \\
band100\_4                                                                            & 100                                       & 1.41                                          & 1.28                                          & 1.56                                          & 1.86                                          & 22.80                                          & 377.78                                         \\
band150\_4                                                                            & 150                                       & 2.61                                          & 5.33                                          & 5.81                                          & 6.95                                          & 63.29                                          & 994.69                                         \\
band200\_4                                                                            & 200                                       & 4.63                                          & 10.77                                         & 10.55                                         & 12.87                                         & 102.51                                         & 1764.86                                        \\
band250\_4                                                                            & 250                                       & 7.33                                          & 16.63                                         & 18.10                                         & 20.68                                         & 153.30                                         & 2616.91                                        \\ \bottomrule
\end{tabular}
\caption{Computation times in seconds of Lagrangian dual approach LD-ASG for different values of $p$ for the max-4-cut problem on Band instances. \label{Tab:band_max4cut_times}}
\end{table}

\begin{figure}[h]
    \centering
    \includegraphics[width=0.45\linewidth]{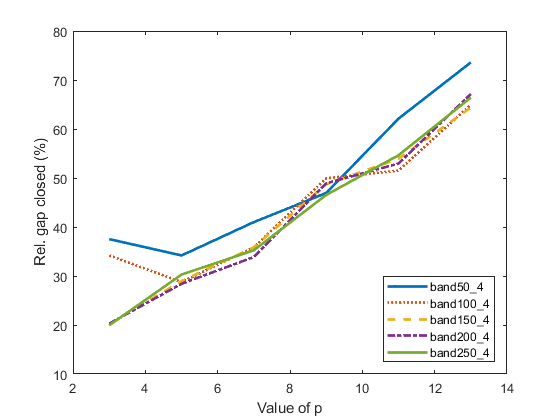}
    \caption{Relative gap closed (in \%) by Lagrangian dual approach LD-ASG for the max-$4$-cut problem on Band instances for different values of $p$. \label{Fig:band_max4cut}}
\end{figure}

\subsection{Computational results for the Spinglass instances} \label{Subsec:numericspinglass}

The final class of instances for which we present results are the Spinglass instances~\cite{DeSimone}.
 In \cite{Ghaddar2011ABA} are listed root node bounds for the Spinglass instances ${\rm spinglass2pm}_{n_t,n_r}$ where  $n_t\in \{4,5,6,7,8,9\}$. Our level $p=7$ bound on those instances close the gap to the optimal solution within a few seconds. Specifically, for the largest instance, it takes approximately 2.8 seconds.

In Table~\ref{Tab:spinglass_maxkcut} we present bounds, optima and relative gap closed
for $p = 3$ and $p = 9$ on a set of the Spinglass instances. We compute Lagrangian dual bounds for the max-2-, max-3- and max-4-cut problem. Optima for those problems and the Spinglass instances are taken from~\cite{Fakhimi2022TheMK}. Also,
computation times for the approach LD-ASG are shown. 

We observe that our Lagrangian dual bounds perform very well on these
special structured instances. Already for $p = 3$, we observe
that the relative gap closed is at least about 40\%,
and in some cases as large as 58\%. For $p = 9$, we observe a significant 
improvement, leading to a relative gap closed within the
range 64-80\%. These bounds can be computed relatively fast,
with computation times below a minute for almost all considered
instances.

\begin{table}[h]
\footnotesize
\centering
\begin{tabular}{@{}cccccccccccc@{}}
\toprule
\multicolumn{1}{c}{\textbf{}}                                                          & \multicolumn{1}{c}{\textbf{}}             & \multicolumn{1}{c}{\textbf{}}             & \multicolumn{1}{c}{\textbf{}}             & \multicolumn{1}{c}{\textbf{}}    & \multicolumn{1}{c}{\textbf{}}    & \multicolumn{3}{c}{\textbf{LD bound for $\mathbf{p = 3}$}}                                                                                                                                                            & \multicolumn{3}{c}{\textbf{LD bound for $\mathbf{p = 9}$}}                                                                                                                                                            \\ \cmidrule(l){7-9} \cmidrule(l){10-12} 
\multicolumn{1}{c}{\textbf{\begin{tabular}[c]{@{}c@{}}Instance \\ Class \\ $(n_t \times n_r)$\end{tabular}}} & \multicolumn{1}{c}{\textbf{$\mathbf{n}$}} & \multicolumn{1}{c}{\textbf{$\mathbf{m}$}} & \multicolumn{1}{c}{\textbf{$\mathbf{k}$}} & \multicolumn{1}{c}{\textbf{SDP-B}} & \multicolumn{1}{c}{\textbf{OPT}} & \multicolumn{1}{c}{\textbf{LD}} & \multicolumn{1}{c}{\textbf{\begin{tabular}[c]{@{}c@{}}Rel. gap\\ closed \\ (\%)\end{tabular}}} & \multicolumn{1}{c}{\textbf{\begin{tabular}[c]{@{}c@{}}Comp.\\ time (s)\end{tabular}}} & \multicolumn{1}{c}{\textbf{LD}} & \multicolumn{1}{c}{\textbf{\begin{tabular}[c]{@{}c@{}}Rel. gap\\ closed \\ (\%)\end{tabular}}} & \multicolumn{1}{c}{\textbf{\begin{tabular}[c]{@{}c@{}}Comp.\\ time (s)\end{tabular}}} \\ \midrule
$12 \times   12$                                                                       & 144                                       & 288                                       & 2                                         & 113.6                            & 104                              & 108.0                           & 57.9                                                                                        & 13.51                                                                                 & 106.5                           & 73.7                                                                                        & 15.16                                                                                 \\
$12 \times 13$                                                                         & 169                                       & 338                                       & 2                                         & 128.1                            & 114                              & 121.1                           & 49.4                                                                                        & 22.92                                                                                 & 118.2                           & 70.4                                                                                        & 23.31                                                                                 \\
$14 \times 14$                                                                         & 196                                       & 392                                       & 2                                         & 150.0                            & 132                              & 141.4                           & 47.8                                                                                        & 35.34                                                                                 & 137.2                           & 71.3                                                                                        & 39.05                                                                                 \\
$15 \times 15$                                                                         & 225                                       & 450                                       & 2                                         & 167.1                            & 146                              & 156.5                           & 50.0                                                                                        & 38.72                                                                                 & 152.0                           & 71.5                                                                                        & 40.78                                                                                 \\
$16 \times 16$                                                                         & 256                                       & 521                                       & 2                                         & 197.2                            & 178                              & 187.1                           & 52.7                                                                                        & 55.11                                                                                 & 183.3                           & 72.6                                                                                        & 63.02                                                                                 \\
$12 \times 12$                                                                         & 144                                       & 288                                       & 3                                         & 129.3                            & 120                              & 124.8                           & 48.5                                                                                        & 18.06                                                                                 & 121.9                           & 79.2                                                                                        & 22.86                                                                                 \\
$12 \times 13$                                                                         & 169                                       & 338                                       & 3                                         & 150.5                            & 138                              & 145.1                           & 43.7                                                                                        & 24.15                                                                                 & 142.4                           & 64.7                                                                                        & 27.64                                                                                 \\
$14 \times 14$                                                                         & 196                                       & 392                                       & 3                                         & 175.6                            & 161                              & 169.2                           & 44.0                                                                                        & 29.07                                                                                 & 164.7                           & 74.8                                                                                        & 28.66                                                                                 \\
$15 \times 15$                                                                         & 225                                       & 450                                       & 3                                         & 197.8                            & 179                              & 189.3                           & 45.3                                                                                        & 36.70                                                                                 & 184.4                           & 71.3                                                                                        & 41.14                                                                                 \\
$16 \times 16$                                                                         & 256                                       & 521                                       & 3                                         & 229.4                            & 211                              & 221.1                           & 45.1                                                                                        & 47.83                                                                                 & 216.5                           & 70.2                                                                                        & 74.92                                                                                 \\
$12 \times 12$                                                                         & 144                                       & 288                                       & 4                                         & 130.2                            & 120                              & 126.2                           & 39.5                                                                                        & 10.84                                                                                 & 123.0                           & 70.6                                                                                        & 14.92                                                                                 \\
$12 \times 13$                                                                         & 169                                       & 338                                       & 4                                         & 152.1                            & 139                              & 146.8                           & 39.9                                                                                        & 10.37                                                                                 & 143.6                           & 64.6                                                                                        & 15.82                                                                                 \\
$14 \times 14$                                                                         & 196                                       & 392                                       & 4                                         & 177.4                            & 162                              & 171.3                           & 39.5                                                                                        & 16.17                                                                                 & 166.6                           & 70.0                                                                                        & 25.75                                                                                 \\
$15 \times 15$                                                                         & 225                                       & 450                                       & 4                                         & 199.8                            & 180                              & 191.7                           & 40.9                                                                                        & 22.24                                                                                 & 185.1                           & 74.3                                                                                        & 35.53                                                                                 \\
$16 \times 16$                                                                         & 256                                       & 512                                       & 4                                         & 231.8                            & 213                              & 224.2                           & 40.3                                                                                        & 39.01                                                                                 & 219.2                           & 67.2                                                                                        & 41.54                                                                                 \\ \bottomrule
\end{tabular}
\caption{Bounds (SDP-B, LD and optimum), relative gap closed by the Lagrangian dual bound and computation times for LD-ASG for the max-$4$-cut problem on the Spinglass instances. \label{Tab:spinglass_maxkcut} }
\end{table}

\section{Conclusion} \label{sect:conclusion}
In this work we consider the Lagrangian duality theory for MISDPs and show its potential to provide strong bounds for problems that can be modeled as a MISDP.

Starting from a MISDP in standard form, we introduce the Lagrangian relaxation and the associated Lagrangian dual problem, see~\eqref{Eq:LD}. Since this problem exploits both positive semidefiniteness and integrality, we show that the resulting Lagrangian dual bound is always at least as strong as the continuous SDP relaxation, see Theorem~\ref{Thm:AlternativeLD} and 
Corollary~\ref{Cor:sandwich}. Also, we consider conditions under which the Lagrangian dual bound attains its maximum and minimum value, see Theorem~\ref{Thm:StrongDual}. For the case of integer SDPs, we extend on the aforementioned theory, and introduce a hierarchy of Lagrangian dual bounds by exploiting the structure of integer PSD matrices, see Table~\ref{Tab:IntegerPSDSets}. The level-$p$ bound is obtained as the Lagrangian dual of an ISDP, see~\eqref{Def:LDforBSDP}, having additional constraints on some of its $p$-by-$p$ submatrices. 

In Section~\ref{Sect:SolvingLagDual} we present three algorithms for solving the Lagrangian dual problem, each of them finding its roots in non-smooth optimization. To the best of our knowledge, the projected-deflected and the projected-accelerated subgradient algorithms, see Algorithms~\ref{Alg:projectedSubgradient} and~\ref{Alg:projectedAcceleratedSubgradient}, are the first implementations of subgradient methods to obtain ISDP bounds. We also present a projected  bundle algorithm in Algorithm~\ref{Alg:projectedBundle}, which differs from bundle methods in the literature. Instead of solving an SDP in each iteration, our algorithm enumerates over a discrete set, which can be done efficiently for small sizes of $p$.

The three proposed algorithms are exploited to obtain Lagrangian dual bounds for the max-$k$-cut problem in Section~\ref{Subsec:MaxCutNumerics}. The Lagrangian dual bounds are stronger than their continuous SDP counterparts on all considered instances and often by a significant amount. 
The bounds become stronger for larger levels in the hierarchy, although the marginal improvement diminishes over $p$. The relative gap closed by the Lagrangian dual bound (with $p = 13$) compared to the continuous SDP bound can be as large as 70\% for certain graph types. When comparing the three algorithms, we observe that the projected-accelerated subgradient algorithm is favored for small values of $p$, whereas the projected bundle algorithm becomes superior when $p$ becomes large. This computational advantage of the bundle algorithm upon the subgradient algorithms is evident from the number of iterations required to achieve the given accuracy. Since the computation times of the three proposed algorithms is moderate, this opens perspectives for exploiting Lagrangian dual bounds in a branching framework.

\typeout{}
\bibliography{QCCP_QTSP}

\newpage
\begin{appendix}

\section{Set partitions and their cardinalities} \label{Appendix:SetPartitions}
In Section~\ref{sect:Hierarchy of Lagrangian dual} we introduced the sets $\mathcal{S}^n_+ (B,r)$, where we claim that for small values of $n$ we can fully enumerate over these sets. This enumeration involves 
the combinatorial structure of so-called set partitions, i.e., partitions of the set of integers from $1$ to $n$. As a result, the cardinalities of $\mathcal{S}^n_+ (B,r)$ can be expressed in terms of some well-known combinatorial quantities. These quantities are presented below.

\paragraph{Stirling number of the second kind}
A Stirling number of the second kind, see e.g.,~\cite{ConwayGuy}, counts the number of possibilities to partition a set of $n$ elements into $k$ non-empty subsets. It is denoted by $\stirlingii{n}{k}$. 
In particular, $\stirlingii{0}{0} = 1$, $\stirlingii{0}{k} = 0$ for all $k > 0$ and $\stirlingii{n}{1} = 1$ for all $n \geq 1$. For higher values of $n$ and $k$, the recurrence relation $\stirlingii{n+1}{k} = k\stirlingii{n}{k} + \stirlingii{n}{k-1}$ can be applied.

\paragraph{Bell number}
The $n$th Bell number, denoted by $B_n$, equals the number of partitions of a set of $n$ elements~\cite{ConwayGuy}. 
We have $B_0 = B_1 = 1$ and higher Bell numbers can be obtained by the recurrence relation $B_{n+1} = \sum_{k=0}^n  {n \choose k }B_k$. Since the Bell number takes into account all possible partitions of a set (irrespective of the number of subsets), the $n$th Bell number is the sum over all Stirling numbers of the second kind with $n$ fixed, i.e.,
\begin{align*}
    B_n = \sum_{k=0}^{n} \stirlingii{n}{k}. 
\end{align*}

\paragraph{$\bold{B}$-type Stirling number of the second kind}
Let us consider partitions of the extended set of integers $\{-n, \ldots, n\}$. We call a partition $\pi$ of $\{-n, \ldots, n\}$ a $B$-type (or signed) partition~\cite{SaganSwanson} if (1) $\pi = -\pi$, i.e., for any subset $B \in \pi$, it holds that $-B = \{-b \, : \, b \in B\} \in \pi$, and (2) $\pi$ contains just one subset $B_0$ such that $B_0 = -B_0$, which is called the 0-subset. For example, the following partition of $\{-5, \ldots, 5\}$ is a $B$-type partition: 
\begin{align*}
    \{\{1,-5\}, \{2\}, \{-4, -3, 0, 3, 4\}, \{-2\}, \{-1,5\}\},
\end{align*}
that consists of 5 subsets. The $B$-type Stirling number of the second kind $S_B(n,k)$ counts the number of $B$-type partitions of $\{-n, \ldots, n\}$ that consist of exactly $k$ pairs of subsets (excluding the 0-subset). Indeed, $S_B(n,k)$ is the $B$-type equivalent of the Stirling number of the second kind $\stirlingii{n}{k}$.

\paragraph{Dowling number}
The $n$th Dowling number, denoted by $D_n$, equals the total number of $B$-type partitions of the set $\{-n, \ldots, n\}$. 
Since the number of subset pairs in any $B$-type partition of $\{-n, \ldots, n\}$ can range from $0$ to $n$, we have
\begin{align*}
    D_n = \sum_{k = 0}^n S_B(n,k).
\end{align*}
Thus, the $n$th Dowling number has the same relationship with the $B$-type Stirling number of the second kind as the $n$th Bell number has with the regular Stirling number of the second kind.

\medskip 
\paragraph{Relationship with $\mathbf{\mathcal{S}_+^n(B,r)}$}
Here, we  couple 
the above-mentioned quantities to the cardinality of the sets $\mathcal{S}^n_+(B,r)$. For $B = \{0,1\}$ and $1 \leq r \leq n$, the elements in $\mathcal{S}^n_+(\{0,1\},r)$ are of the form $X = \sum_{i=1}^r x_ix_i^\top$, where $x_i \in \{0,1\}^n$ and each $x_i$ and $x_j$ have non-overlapping support.
Each $x_i$ can be thought of as the characteristic vector of a subset of $\{1, \ldots, n\}$. If we add to this set an additional element $0$, an element $X \in \mathcal{S}^n_+(\{0,1\},r)$ can be seen as a partition of the set $\{0, \ldots, n\}$, where the set that contains the $0$-element corresponds to all integers that are not in the support of one of the $x_i$'s. Hence, there exists a bijection between the elements in $\mathcal{S}^n_+(\{0,1\},r)$ and all partitions of a set of $n+1$ elements that consists of at most $r+ 1$ subsets. We conclude
\begin{align*}
    |\mathcal{S}^n_+(\{0,1\},r)| = \sum_{k = 0}^{r+1} \stirlingii{n+1}{k}. 
\end{align*}
In particular, when $r = n$, we have 
\begin{align*}
    |\mathcal{S}^n_+(\{0,1\})| = |\mathcal{S}^n_+(\{0,1\},n)| = \sum_{k = 0}^{n+1} \stirlingii{n+1}{k} = B_{n+1}. 
\end{align*}
Let us consider $B = \{0, \pm 1\}$. The elements in $\mathcal{S}^n_+(\{0,\pm 1\},r)$ are of the form $X = \sum_{i=1}^r x_ix_i^\top$, where $x_i \in \{0,\pm 1\}^n$ and each $x_i$ and $x_j$ have non-overlapping support. Each $x_i$ can be seen as the characteristic vector of a subset of $\{-n, \ldots, n\}$, where we include $-u$ if $(x_i)_u = -1$ and $u$ if $(x_i)_u = 1$ (and we do not include $u$ or $-u$ if $(x_i)_u = 0$). Observe that the characteristic vector $-x_i$ leads to the \textit{same} partition of $\{-n, \ldots, n\}$. A matrix $X = \sum_{i=1}^r x_ix_i^\top$ then corresponds to a partition of $\{-n, \ldots, n\}$, where the elements that are not in the support of any $x_i$ are in the same subset as~$0$. Therefore, each matrix $X$ corresponds to a $B$-type partition of $\{-n, \ldots, n\}$. Thus, there exists a bijection between the elements in $\mathcal{S}^n_+(\{0,\pm 1\},r)$ and all $B$-type partitions of the set $\{-n, \ldots, n\}$ consisting of at most $r$ subset pairs. We conclude
\begin{align*}
    |\mathcal{S}^n_+(\{0,\pm 1\},r)| = \sum_{k = 0}^{r} S_B(n,k). 
\end{align*}
For $r = n$, we obtain
\begin{align*}
     |\mathcal{S}^n_+(\{0,\pm 1\})| = |\mathcal{S}^n_+(\{0,\pm 1\},n)| = \sum_{k = 0}^{n} S_B(n,k) = D_n. 
\end{align*}
We finalize this section by presenting how the cardinalities of the above-mentioned sets evolve over $n$. Table~\ref{Tab:CardSetPartitions} contains the values of $|\mathcal{S}^n_+(B,r)|$ for different $B$, $r$ and $n$, see also~\cite{OEIS}. 

\begin{table}[H]
\centering
\small
\begin{tabular}{@{}lrrrrrrr@{}}
\toprule
\multicolumn{1}{c}{\textbf{Set}}           & \multicolumn{7}{c}{\textbf{Cardinality}}                                                                                                                                                                                 \\
\multicolumn{1}{c}{}              & \multicolumn{1}{c}{$n = 3$} & \multicolumn{1}{c}{$n = 4$} & \multicolumn{1}{c}{$n = 5$} & \multicolumn{1}{c}{$n = 6$} & \multicolumn{1}{c}{$n = 7$} & \multicolumn{1}{c}{$n = 8$} & \multicolumn{1}{c}{$n = 9$} \\ \midrule
$|\mathcal{S}_+^n(\{0,1\}, 3)|$   & 15                          & 51                          & 187                         & 715                         & 2795                        & 11051                       & 35550                       \\
$|\mathcal{S}_+^n(\{0,1\})|$      & 15                          & 52                          & 203                         & 877                         & 4140                        & 21147                       & 115975                      \\
$|\mathcal{S}_+^n(\{0,\pm1\},3)|$ & 24                          & 115                         & 622                         & 3656                        & 22724                       & 146565                      & 968922                      \\
$|\mathcal{S}_+^n(\{0,\pm1\})|$   & 24                          & 116                         & 648                         & 4088                        & 28640                       & 219920                      & 1832224                     \\ \bottomrule
\end{tabular}
\caption{Number of matrices in $\mathcal{S}_+^n(B,r)$ for different $B$, $r$ and $n$. \label{Tab:CardSetPartitions} }
\end{table}

\end{appendix}
\end{document}